\documentclass[12pt,leqno,a4paper]{article}
\usepackage{hyperref}
\usepackage{amsmath}
\usepackage{amssymb}
\usepackage{amsthm}
\usepackage{graphicx,epstopdf}

\makeatletter
\newcommand{\guio}[1]{\nobreakdash-\hspace{0pt}#1}

\long\def\@makecaption#1#2{%
  \vskip\abovecaptionskip
  \sbox\@tempboxa{#1\ #2}%
  \ifdim \wd\@tempboxa >\hsize
    #1\ #2\par
  \else
    \global \@minipagefalse
    \hb@xt@\hsize{\hfil\box\@tempboxa\hfil}%
  \fi
  \vskip\belowcaptionskip}
  \makeatother

\newtheorem{theorem}{Theorem}
\newtheorem{lemma}{Lemma}
\newtheorem{defi}{Definition}

\theoremstyle{definition}
\newtheorem*{example}{Example}
\newtheorem*{gracies}{Acknowledgements}

\newcommand{\Rd}{\mathbb{R}^d}
\newcommand{\C}{\mathbb{C}}
\newcommand{\Capa}{\operatorname{Cap}}
\newcommand{\pv}{\operatorname{p.v.}}
\title{Differentiability properties of Riesz potentials of finite measures and non-doubling Calder\'on--Zygmund theory}

\author{Juli\`a Cuf\'{\i} and Joan Verdera}

\date{}

\begin{document}

\footnotetext{\\*[-11pt]
}

\maketitle

\begin{abstract}
We study differentiability properties of the Riesz potential, with kernel of homogeneity $2-d$ in $\mathbb{R}^{d}$, $d\ge 3$, of a finite Borel measure. In the plane we consider the logarithmic potential of a finite Borel measure.
We introduce a notion of differentiability in the capacity sense, where capacity is
Newtonian capacity in dimension $d \ge 3$ and Wiener capacity in the plane.
We require that the first order remainder at a point is small when measured by means of a normalized weak capacity ``norm'' in balls of small radii centered at the point. This implies $L^{p}$ differentiability in the Calder\'on--Zygmund sense for $1\le p < d/d-2$. If $d \ge 3$,\linebreak we show that the Riesz potential of a finite Borel measure is differentiable in the capacity sense except for a set of zero $C^1$-harmonic capacity. The result is sharp and depends
on deep results in non-doubling Calder\'on--Zygmund theory. In the plane the situation is different.
Surprisingly there are two distinct notions of differentiability in the capacity sense. For each of them we obtain the best possible result on the size of the exceptional set in terms of Hausdorff measures.  We obtain, for $d \ge 3$, results on Peano second order differentiability in the sense of capacity with exceptional sets of zero Lebesgue measure. Finally, as an application, we find a new proof of the well-known fact that the equilibrium measure is singular with respect to Lebesgue measure.
\end{abstract}

\section{Introduction}

Calder\'{o}n and Zygmund applied  their celebrated results on singular integrals to
understand differentiability properties of functions defined on
subsets of $\Rd.$ Besides the foundational paper \cite{CZ1}, where
logarithmic potentials in the plane and Riesz potentials in higher
dimensions were considered, one may consult \cite{CZ2} and the book
\cite{S}, in which the central results known up to the seventies
were presented. A recent interesting paper on the subject is \cite{ABC}. The
setting for our results is as follows.

Let  $\mu$  be a Borel finite measure in $\Rd$ and, in dimension $d
\ge 3,$ consider the Riesz potential
\begin{equation}\label{pot}
u(x)=(P\mu)(x)= \int_{\Rd} \frac{1}{|x-y|^{d-2}}\,d\mu(y),\quad
x\in\mathbb{R}^d.
\end{equation}
In the plane we take the logarithmic potential
\begin{equation}\label{potlog}
u(z)=(P\mu)(z)= \int_{\C} \log \frac{1}{|z-w|}\,d\mu(w),\quad
z\in\C.
\end{equation}
The kernel chosen in all dimensions $d \ge 2$ is a constant multiple of the
fundamental solution of the Laplacian in $\Rd.$ The distributional
gradient of the potential $u$ is
\begin{equation}\label{gradpot}
\nabla u = -(d-2)\,\frac{x}{|x|^d}* \mu, \quad d \ge 3,
\end{equation}
and
\begin{equation*}\label{gradpot2}
\nabla u = - \frac{z}{|z|^2}*\mu = -\,\frac{1}{\overline{z}}*\mu,
\quad d =2.
\end{equation*}
Since the kernel in the preceding identities is locally integrable,
$\nabla u$ is a locally integrable function, hence well defined a.e.
The second derivatives of $u$ in the sense of distributions are
given in dimension $d >2$ by
\begin{equation}\label{secondpot}
\partial_{jj} u =
-(d-2)\;\operatorname{p.v.} \frac{|x|^2-d\,
x_j^2}{|x|^{d+2}}
*\mu   -\frac{1}{d} (d-2) \omega_{d-1} \; \mu, \quad 1 \le j \le d,
\end{equation}
where  $\omega_{d-1}$ is the
$d-1$-dimensional surface measure of the unit sphere in $\Rd,$ and
\begin{equation}\label{secondpotcross}
\partial_{jk} u = d(d-2)\,\operatorname{p.v.} \frac{x_j x_k}{|x|^{d+2}} *\mu, \quad 1 \le j \neq k \le d.
\end{equation}
 In
dimension $d=2$, setting $z=x+i y$, one gets
\begin{equation}\label{secondpot2}
\frac{\partial^2}{\partial x^2} u = \operatorname{p.v.} \frac{x^2-
y^2}{|z|^{4}} *\mu - \pi \mu,
\end{equation}
\begin{equation}\label{secondpot22}
\frac{\partial^2}{\partial y^2} u = \operatorname{p.v.} \frac{y^2-
x^2}{|z|^{4}} *\mu - \pi \mu,
\end{equation}
and
\begin{equation}\label{secondpot23}
\frac{\partial^2}{\partial x\,\partial y} u = \operatorname{p.v.} \frac{2 x
y}{|z|^{4}} *\mu.
\end{equation}

The key fact is that one can give at almost all points a sense to
all right hand sides in \eqref{secondpot}--\eqref{secondpot23}.
Indeed, the principal value singular integrals exist a.e.\ after the
results of \cite{CZ1} and one can assign to the measure $\mu$ in
\eqref{secondpot}, \eqref{secondpot2} and \eqref{secondpot22} at the
point $x$ the density
$$\lim_{r \rightarrow 0} \frac{\mu(B(x,r))}{|B(x,r)|},$$
which exists a.e.\ in $\Rd.$ We denote by $|E|$ the $d$-dimensional
Lebesgue measure of the measurable set $E.$ Calder\'{o}n and Zygmund
proved that if $d=2$ and $\mu$ is absolutely continuous with density
locally in $\operatorname{LlogL},$ then  $u$ has a second
differential in the Peano sense a.e. One obtains the same conclusion
for $d\ge 3$ if $\mu$ is assumed to be absolutely continuous with
density in $L^q(\Rd), \; q
> d/2.$
Recall that a function $u$ defined in a neighborhood of a point $a$
has a second differential in the Peano sense if there exists
constants $A_i, \; 1 \le i \le d,$ and $B_{jk}, \; 1 \le j, k \le
d,$ such that
\begin{equation*}
u(x)=u(a)+\sum_{i=1}^d A_i (x_i-a_i)+ \sum_{j,k =1}^d B_{jk}
(x_j-a_j)(x_k-a_k)+ \varepsilon(|x-a|)\,|x-a|^2
\end{equation*}
for a certain function $\varepsilon(t)$ which tends to $0$ with $t$.

Brilliant work by many people during the last decade has shown that
most of Calder\'{o}n--Zygmund theory holds in very general contexts in
which the classical homogeneity assumption is dropped. It is enough
that  the underlying measure $m$ be a positive locally finite Borel
measure in $\Rd$ satisfying a growth condition
$$m(B(x,r)) \le C \,  r^n, \quad 0 < r < R, $$
$R$  being the diameter of the support of $m$ and $0<n\le d$. Hence
$m$ is not necessarily doubling.  See, for instance, \cite[Chapter
2]{T2} and the many references given there. It appears  then
appropriate to explore what new differentiability results might the
general non-doubling Calder\'{o}n--Zygmund theory make available. We
consider a variant of the notion of differentiability in the $L^p$
sense in which we require the remainder to tend to zero in the weak
capacitary ``norm".

\begin{defi}\label{capdif}
Let $u$ be a real function defined in a neighborhood of a point $a \in \Rd.$  Given real numbers $A_1, \dotsc, A_d$ set
\begin{equation*}
Q(x) = \frac{|u(x)-u(a) - \sum_{i=1}^d A_i (x_i-a_i)|}{|x-a|}.
\end{equation*}
We say that $u$ is differentiable in the capacity sense at the point
$a$ provided there exist real numbers $A_1, \dotsc, A_d$ such that
\begin{equation}\label{defdifcap}
\lim_{r \rightarrow 0} \frac{\sup_{t > 0} t \operatorname{Cap}(\{x \in B(a,r) : Q(x) > t \})}{\operatorname{Cap}(B(a,r))} = 0.
\end{equation}
Here $\Capa$ stands for Wiener capacity in the plane and Newtonian capacity associated with the kernel $1/|x|^{d-2}$
 in higher dimensions. See section 2 for  precise definitions.
\end{defi}

\pagebreak

In other words, we require that the normalized weak capacity norm in the ball $B(a,r)$ of the quotient $Q(x)$ tends to $0$ with $r$.
 This makes sense for potentials $u=P(\mu) $ of finite Borel measures, because they satisfy the inequality
\begin{equation}\label{wcap}
\operatorname{Cap}(\{x \in \Rd: P(\mu)(x) > t \}) \le \frac{\| \mu \|}{t}, \quad 0 < t,
\end{equation}
and so, in particular they are defined except for a set of capacity zero.

The notion of differentiability in the capacity sense can be
weakened by replacing the denominator $|x-a|$ in $Q(x)$ by $r$ and
then rescaling $t.$ We get the following.

\begin{defi}\label{wcapdif}
Let $u$ be a real function defined in a neighborhood of a point $a
\in \Rd.$ We say that $u$ is differentiable in the weak capacity
sense at the point $a$  provided there exist real numbers $A_1, \dotsc,
A_d$ such that
\begin{equation*}
\lim_{r \rightarrow 0} \frac{\sup_{t > 0} t \operatorname{Cap}(\{x
\in B(a,r) : |u(x)-u(a)-\sum_{i=1}^d A_i (x_i-a_i)| > t
\})}{r\,\operatorname{Cap}(B(a,r))} = 0.
\end{equation*}
\end{defi}

A simple argument, which consists in expressing a ball as a union of
dyadic annuli, gives readily that the above two notions of
differentiability coincide if $d \ge 3.$ Instead they are different
in the plane as we will discuss later. This is due to the fact that
$\Capa(B(a,r)) = \frac{1}{\log \frac{1}{r}},$ in the plane, while in dimensions $d \ge 3$
the dependence of the capacity of a ball on the radius is via a power: $\Capa(B(a,r)) = c_d \, r^{d-2}.$

Our first result
concerns differentiability in the capacity sense of Riesz potentials
of finite measures in $\Rd, \;d \ge 3.$

\begin{theorem}\label{teo1}
For a positive finite Borel measure $\mu$ in $\Rd, \; d \ge 3,$ the Riesz
potential
\begin{equation*}
u(x)=\int_{\Rd} \frac{1}{|x-y|^{d-2}}\,d\mu(y),\quad
x\in\mathbb{R}^d,
\end{equation*}
is differentiable in the capacity sense at the point $a \in \Rd$ if and only if
\begin{equation}\label{densitymu}
\lim_{r \rightarrow 0} \frac{\mu(B(a,r))}{r^{d-1}} =0
\end{equation}
and the principal value
\begin{equation}\label{pvmu}
\pv \int \frac{a-y}{|a-y|^d} \,d\mu(y) = \lim_{\varepsilon \rightarrow 0} \int_{|y-a|>\varepsilon}
\frac{a-y}{|a-y|^d}
\,d\mu(y)
\end{equation}
exists.
\end{theorem}

Hence differentiability in the capacity sense is exactly equivalent to vanishing of the $(d-1)$-dimensional
density of the measure and existence of the principal value \eqref{pvmu}. Existence of the principal
values brings into the picture singular integrals with respect to a non-doubling underlying measure. This
will happen when one is dealing with measures~$\mu$ which are as spread as the vanishing of the
$(d-1)$-dimensional density allows. In this case the kernel ${a-y}/{|a-y|^d}$ will behave as a singular
Calder\'on--Zygmund kernel with respect to $\mu$ and there is no reason to expect $\mu$ to be doubling.

Our second result asserts that the Riesz potential of each finite Borel measure $\mu$ in~$\Rd, \; d\ge3$
is differentiable in the capacity sense except in a set whose size is controlled by an appropriate set
function. This set function is called $C^1$ harmonic capacity
and is  defined as follows.

The $C^1$ harmonic capacity of a compact set $E \subset \Rd$
is
\begin{equation}\label{kappa}
\kappa^c(E)= \sup |\langle T,1 \rangle|
\end{equation}
where the supremum is taken over those distributions $T$ supported
on $E$ such that\linebreak $T*x/|x|^{d}$ is a continuous vector valued function on $\Rd$
satisfying $\|(T*x/|x|^{d})(x)\|\! \le\! 1$, $x \in \Rd.$ The terminology refers to the fact that
convolving such a distribution with the fundamental solution of the Laplacian one gets a harmonic function
on $\Rd \setminus E$ of class~$C^1(\Rd)$.

It is readily seen that $\kappa^c(E)=0$ if and only each function of
class $C^1(\Rd)$ harmonic on $\Rd \setminus E$ is linear. The
homogeneity of the set function $\kappa^c$ is $d-1$, that is,
$\kappa^c(\lambda E)=\lambda^{d-1} \kappa^c( E).$ Deep results of
\cite{RT} show that $C^1$ harmonic capacity can be described in
terms of positive measures supported on the set, having null
$(d-1)$-dimensional density and enjoying the property that the
singular integral operator determined by the vectorial kernel
$x/|x|^d$ is bounded on the $L^2$ Lebesgue space of the measure. The
description is rather explicit and in particular shows that
$\kappa^c$ is semiadditive, i.\ e.,
$$\kappa^c(E \cup F) \le C  \, \left(\kappa^c(E)+\kappa^c(F)\right),$$
for a dimensional constant~$C$ independent of the compact sets $E$ and $F.$
If $F$ is an arbitrary subset of $\Rd$, then $\kappa^c(F)$ is defined as the supremum of  $\kappa^c(E)$ over
all compact subsets $E$ of $F$.
See section 2 for more details.

\begin{theorem}\label{teo2}
For each finite Borel measure $\mu$ in $\Rd, \; d \ge 3,$ the Riesz
potential
\begin{equation*}
u(x)=\int_{\Rd} \frac{1}{|x-y|^{d-2}}\,d\mu(y),\quad
x\in\mathbb{R}^d,
\end{equation*}
is differentiable in the capacity sense at $\kappa^c$ almost all
points.
\end{theorem}
The result is sharp. In fact the unit sphere in $\Rd$ can be shown
to be the set of points at which the potential of a positive finite Borel measure
is not differentiable in the capacity sense. The maximal dimension
of a set of vanishing $\kappa^c$ capacity is $d-1$ and the sphere
is, in some sense, the biggest such set that one can imagine.

Somehow surprisingly, in dimension $d=2$ the result one finds is different.
This is due to the usual difficulties associated with the
logarithmic kernel. Indeed, there are two separate results, each
dealing with one of the two notions of differentiability in the
capacity sense we have in the plane.

Let $\varphi$ be the measure function
\begin{equation}\label{fi}
\varphi(t)=\begin{cases}  t \, \frac{1}{\log \frac{1}{t}},  &0
<t \le e^{-1}, \\*[5pt]
t, &e^{-1} \le t,
\end{cases}
\end{equation}
and $H^\varphi$ the associated Hausdorff measure. We then have

\begin{theorem}\label{teo3}
For each finite Borel measure $\mu$ in $\C$ the logarithmic
potential
\begin{equation*}
u(z)=\int_{\C} \log \frac{1}{|z-w|}\,d\mu(w),\quad z \in \C,
\end{equation*}
is differentiable in the weak capacity sense at $H^\varphi$ almost
all points.
\end{theorem}
The result is sharp in the scale of Hausdorff measures. Given a
measure function~$\Phi$ with the property that $\Phi(r)/ \varphi(r)
\rightarrow \infty $ as $r \rightarrow 0$ and satisfying another
minor assumption,  then there exists a finite Borel measure whose
logarithmic potential is not differentiable in the weak capacity
sense on a set of positive $H^\Phi$-measure. See Theorem \ref{teo6}
in section 5 for a precise statement.

The next result deals with differentiability in the capacity sense
in the plane. Let $\psi$ stand for the measure function
\begin{equation}\label{psi}
\psi(t)=\begin{cases}
t \, \frac{1}{\log^2( \frac{1}{t})}, &0 <t \le e^{-1}, \\*[5pt]
t, &e^{-1} \le t.
\end{cases}
\end{equation}

\begin{theorem}\label{teo4}
For each finite Borel measure $\mu$ in $\C$ the logarithmic
potential
\begin{equation*}
u(z)=\int_{\C} \log \frac{1}{|z-w|}\,d\mu(w),\quad z \in \C,
\end{equation*}
is differentiable in the capacity sense at $H^\psi$ almost all
points.
\end{theorem}

As before, the result is sharp in the scale of Hausdorff measures.
Given a measure function $\Psi$ with the property that $\Psi(r)/
\psi(r) \rightarrow \infty $ as $r \rightarrow 0$ and satisfying
another minor assumption,  then there exists a finite Borel measure
whose logarithmic potential is not differentiable in the capacity
sense on a set of positive $H^\Psi$-measure. See Theorem \ref{teo7}
in section 6 for a precise statement.

We turn now to Peano second order differentiability. Given a
function $u$, a point $a \in \Rd$ and real numbers $A_i, \; 1 \le i
\le d$, and $B_{jk}, \; 1 \le j, k \le d,$ set
\begin{equation}\label{Q2}
D(x)=\frac{1}{|x-a|^2} \,\left| u(x)-u(a)-\sum_{i=1}^d A_i
(x_i-a_i)- \sum_{j,k =1}^d B_{jk} (x_j-a_j)(x_k-a_k)\right|.
\end{equation}

\begin{defi}\label{capdif2}
Let $u$ be a real function defined in a neighbourhood of a point $a
\in \Rd.$ We say that $u$ is differentiable of the second order in
the capacity sense at the point $a$ provided there exist real
numbers $A_i \; 1 \le i \le d$  and $B_{jk}, \; 1 \le j, k \le d,$ such
that
\begin{equation}\label{defdifcap2}
\lim_{r \rightarrow 0} \frac{\sup_{t > 0} t \operatorname{Cap}(\{x
\in B(a,r) : D(x) > t \})}{\operatorname{Cap}(B(a,r))} = 0.
\end{equation}
\end{defi}

We also mention for the record that there is the corresponding notion of second order differentiability in the
weak capacity sense,
which consists in requiring that
\begin{equation*}
\lim_{r \rightarrow 0} \frac{\sup_{t > 0} t \operatorname{Cap}(\{x
\in B(a,r) : \tilde{D}(x) > t \})}{r^2 \,\operatorname{Cap}(B(a,r))} = 0,
\end{equation*}
where $\tilde{D}(x) = \left| u(x)-u(a)-\sum_{i=1}^d A_i
(x_i-a_i)- \sum_{j,k =1}^d B_{jk} (x_j-a_j)(x_k-a_k)\right|.$

\begin{theorem}\label{teo5}

\begin{enumerate}
\item[(i)]  For each finite Borel measure $\mu$ in $\Rd, \; d \ge 3,$ the
Riesz potential
\begin{equation*}
u(x)=\int_{\Rd} \frac{1}{|x-y|^{d-2}}\,d\mu(y),\quad
x\in\mathbb{R}^d,
\end{equation*}
is differentiable of the second order in the capacity sense at
almost all points (with respect to Lebesgue measure in $\Rd$).
\item[(ii)] There exists a finite Borel measure in $\C$ such that
the logarithmic potential of $\mu$ is not differentiable of the
second order in the weak capacity sense at almost all points of $\Rd.$
\end{enumerate}
\end{theorem}
The preceding result could have been proved in the sixties and its proof
follows standard arguments from \cite{S} for part (i) and an idea of Calder\'{o}n from
\cite{C} for part (ii). In section 8 we apply Theorems \ref{teo4} and \ref{teo5} to provide a new proof of the well-konwn fact that the equilibrium measure is singular with respect to Lebesgue measure.

In \cite{ABC} one proves that the Riesz potential $1/|x|^{d-1} *
\mu$ of a finite Borel measure~$\mu$  is differentiable in the $L^p$
sense, $1 \le p < d/(d-1)$, at almost all points. One can also adopt
in this context our notion of differentiability in the capacity
sense, where this time the capacity involved is the one related to
the kernel $1/|x|^{d-1} $, say $C_{d-1}$. Since one has $C_{d-1}(E)
\ge c\,|E|^{(d-1)/d}$, it turns out that differentiability in the
capacity sense implies differentiability in the $L^p$ sense for the
range $1 \le p < d/(d-1)$. The argument for the proof of Theorem
\ref{teo5} can be adapted easily to obtain differentiability in the
$C_{d-1}$\guio{capacity} sense almost everywhere. Thus one has a
slightly better result.

The paper is organized as follows. In section 2 we collect a series
of background facts on capacities, singular integral operators on
subsets of $\Rd$, and Cantor sets. In sections 3 and~4 we prove Theorems \ref{teo1} and \ref{teo2}, respectively.
In section 5 we prove  Theorem~\ref{teo3}. Sharpness of Theorem \ref{teo3} is established in Theorem \ref{teo6} by means of a construction,
inspired by work of Calder\'{o}n in \cite{C}.   Section 6 is
devoted to Theorem \ref{teo4} and its sharpness, established in Theorem \ref{teo7}. The proof of
Theorem \ref{teo5} is in section 7 and the application to equilibrium measure is in section 8.

Our terminology and notation are standard. For instance, we use the
letter $C$ to denote a positive constant, which may vary at each
occurrence, and which is independent of the relevant parameters.
Usually $C$ depends only on dimension. We use the symbol~$ A \simeq
B$ to indicate that for some constant $C>1$ one has $C^{-1}\,B \le A
\le C \, B$.

\section{Background facts}

\subsection{Wiener and Newtonian capacities.}

If $E$ is a compact subset of $R^d,\; d \ge 3$, the Newtonian capacity of $E$ is
\begin{equation}\label{newtoniancap}
 \Capa(E) = \sup \mu(E)
\end{equation}
where the supremum is taken over all positive finite Borel measures supported on $E$
such that the Riesz potential $P(\mu)$ of $\mu$ satisfies $P(\mu)(x)\le 1, \,x \in \Rd$.
There is an equivalent definition involving the notion of energy. The energy of a measure $\nu$ is
\begin{equation}\label{energy}
 V(\nu)= \iint \frac{1}{|x-y|^{d-2}}\, d\nu(x) \,d\nu(y)
\end{equation}
and one has
\begin{equation}\label{capenergy}
 \Capa(E) = (\inf \{V(\nu) :  \text{support of}\; \nu \subset E \;\text{and} \; \|\nu\|=1\})^{-1}.
\end{equation}
It can easily be seen that $\Capa(B(a,r)) = c_d\,r^{d-2}$ (see \cite{AG}).

In the plane one would like to make the same definitions with the Riesz kernel replaced by the logarithm.
The difficulty is that the kernel changes sign and this causes inconveniences. One way to proceed is
to consider only subsets of the disc centered at the origin of radius $1/2$, so that $|z-w| \le 1$ and
$\log \frac{1}{|z-w|} \ge 0.$ Then the Wiener capacity is \eqref{newtoniancap} with the kernel $1/|x-y|^{d-2}$
replaced by $\log \frac{1}{|z-w|}$. The energy of a measure is \eqref{energy} with the same change in the
kernel. The relation \eqref{capenergy} holds true. We have
\begin{equation*}
 \Capa(B(a,r)) = \frac{1}{\log \frac{1}{r}}, \quad |a|< 1/4, \quad 0 < r < 1/4.
\end{equation*}

Note that the definition of $C^1$ harmonic capacity is similar in
structure to that of Wiener or Newtonian capacities. In
\eqref{kappa} the supremum is taken on all distributions with
support in the set $E$ whose potential satisfies a certain
inequality and in \eqref{newtoniancap} only positive measures are
considered. This is a minor difference: one can show that $\Capa(E)$
is the supremum of $|\langle T, 1 \rangle|$ over all
distributions~$T$ supported on $E$ such that the potential
$1/|x-y|^{d-2}* T$ is a function in $L^\infty(\Rd)$ with norm
bounded by $1$ (see, for instance, \cite{V}). The essential
difference lies in the fact that the kernel involved in the
definition of $C^1$ harmonic capacity  is vectorial and each of its
components is a kernel of variable sign. Then subtle cancellation
phenomena have to be taken into account, which explains the enormous
difficulties arising in the study of $C^1$ harmonic capacity. See
subsection 2.3 below.

\subsection{Singular integrals on subsets of \boldmath$\Rd.$}
Let $m$ be a positive finite Borel measure. Set, for $f \in L^2(m)$,
\begin{equation*}
R_\varepsilon(f m )(x)= \int_{|y-x|>\varepsilon} \frac{x-y}{|x-y|^d}
f(y)\,dm(y), \quad x \in \Rd, \quad \varepsilon > 0.
\end{equation*}
We say that the operator $R$ with kernel $x/|x|^d$ is bounded on
$L^2(m)$ if there exists a constant $C$ such that
\begin{equation}\label{Rfitat}
\int | R_\varepsilon(f m)(x)|^2 \, dm(x) \le C\,  \int |f(x)|^2\,dm(x),
\quad \varepsilon > 0.
\end{equation}
In other words, the truncated operators $R_\varepsilon$ are uniformly
bounded in $L^2(m)$. If $m$ has no atoms, then a necessary condition
for boundedness is the growth condition
\begin{equation*}
m(B(x,r)) \le C \, r^{d-1}, \quad x \in \Rd, \quad 0 < r.
\end{equation*}
Our differentiability theorems depend on the existence of the
principal values
\begin{equation}\label{pv}
\pv \int \frac{x-y}{|x-y|^d} \,dm(y) = \lim_{\varepsilon \rightarrow 0}
R_{\varepsilon}(m)(x).
\end{equation}
In classical Calder\'{o}n--Zygmund theory existence of principal values
is a consequence of the $L^2$ estimate \eqref{Rfitat}, but in the
non-doubling context we are considering existence of principal
values is a much subtler issue. A general result which applies to
our situation appeals  to the vanishing of $(d-1)$-dimensional
density, that is,
\begin{equation}\label{zerodensity}
\lim_{r \rightarrow 0} \frac{m(B(x,r))}{r^{d-1}} = 0, \quad x \in
\Rd.
\end{equation}
It was proven in \cite{MV} that \eqref{Rfitat} and
\eqref{zerodensity} imply existence of the principal values
\eqref{pv} $m$-a.e. This in turn yields, by classical Calder\'{o}n--Zygmund theory arguments, the $m$ a.e.
existence of the principal values
\begin{equation*}
\pv \int \frac{x-y}{|x-y|^d} \,d\nu(y) = \lim_{\varepsilon \rightarrow 0} \int_{|y-x|>\varepsilon} \frac{x-y}{|x-y|^d}
\,d\nu(y),
\end{equation*}
for each finite Borel measure $\nu$.

\subsection{\boldmath$C^1$ harmonic capacity.}  Consider the quantity
\begin{equation*}
\kappa^o(E) = \sup m(E)
\end{equation*}
where the supremum is taken over all positive finite Borel measures $m$ supported on~$E$ such that
$m(B(x,r)) \le r^{d-1}, \; x \in \Rd, \; 0< r,$ $\lim_{r \rightarrow 0} \frac{m(B(x,r))}{r^{d-1}}=0, \;x \in \Rd,$ and
the operator $R$ is bounded on $L^2(m)$ with constant $1$ (that is, \eqref{Rfitat} holds with $C=1$).
In~\cite{RT} one shows that there exists a constant depending only on dimension such that
\begin{equation}\label{kappakappaop}
C^{-1} \,\kappa^o(E) \le  \kappa^c(E) \le C\, \kappa^o(E)
\end{equation}
for all compact sets $E \subset \Rd$. This is a deep result, depending on previous work of Tolsa on semi-additivity of analytic capacity in the plane. In fact, $C^1$ harmonic capacity in the plane turns out to be comparable to continuous analytic capacity. We will use \eqref{kappakappaop} in combination with the previous subsection to conclude that if $\kappa^c(E) > 0$ then there exists a non-zero finite Borel measure $m$ supported on $E$ with zero $(d-1)$-dimensional density for which the principal values  \eqref{pv} exist $m$-a.e.

\subsection{\boldmath$L^2$ boundedness of \boldmath$R$.} There is a non-trivial sufficient condition for boundedness of the operator $R$ on $d=2$ found by Mattila in \cite{M2}. If a positive finite Borel measure $m$ in $\C$ satisfies the growth condition
\begin{equation*}
m(B(x,r)) \le C\, \varphi(r)  , \quad 0 < r < e^{-1},
\end{equation*}
where $\varphi$ is the function \eqref{fi}, then the operator $R$ is bounded on $L^2(m).$ The proof of a
more general result is a
calculation based on Menger curvature. This will be used in combination with the result of subsection 2.2 to conclude that the principal values in \eqref{pv} exist $m$-a.e.

\subsection{Cantor sets.} Along the paper we will make a couple of constructions to show sharpness of our theorems, which
involve planar Cantor sets. Now we recall the definition. Take a sequence $(\lambda_n)_{n=1}^\infty$ such that
$0< \lambda_n < 1/2.$ Start with the unit square $Q_0= [0,1]\times [0,1]$. Take $4$ squares contained in
$Q_0$, with sides of length $\lambda_1$ parallel to the coordinate axis, each with a vertex in common
with $Q_0$. Repeat the operation in each of these $4$ squares with the dilation factor $\lambda_2$ in place of
$\lambda_1$. We obtain $16$ squares of side length $\lambda_1 \lambda_2$. Proceeding inductively we get at
the $n$-th generation $4^n$ squares $Q_j^n, \; 1 \le j \le 4^n,$ of side length
$\sigma_n=\prod_{k=1}^n \lambda_k$. Define the Cantor set associated with the sequence $(\lambda_n)_{n=1}^\infty$
to be $K=\cap_{n=1}^\infty \left(\cup_{j=1}^{4^n} Q_j^n\right).$ There is a unique Borel measure $\mu$
supported on $K$ such that $\mu(Q_j^n)=1/4^n$ for all $j$ and~$n$. This measure plays the role of
canonical measure on the Cantor set.

There is a special family of Cantor sets $K_\beta$, which is worthwhile keeping in mind as a working example, depending on a parameter $\beta \ge 0$,
associated with the sequence
\begin{equation*}
\lambda_n = \frac{1}{4}(1+\frac{\beta}{k}), \quad k=1, 2, \dotsc
\end{equation*}
With this choice of $\lambda_n$ we have
\begin{equation*}
\sigma_n \simeq \frac{n^\beta}{4^n}.
\end{equation*}
The canonical measure $\mu$ on $K_\beta$ satisfies the growth condition
\begin{equation}\label{growthmu}
\mu(B(z,r)) \simeq r \,\frac{1}{\log^\beta(\frac{1}{r})}, \quad z \in K_\beta, \quad 0 < r < 1.
\end{equation}
Note that the function on the right hand side of \eqref{growthmu} is $\varphi(r)$ for $\beta =1$ and $\psi(r)$
for $\beta=2$.

For $\beta =0$ we get the famous ``corner quarters'' Cantor set, which has positive finite length but zero
analytic capacity. For $\beta > 0$ the corresponding Cantor set $K_\beta$ is a compact set of infinite length
and Hausdorff dimension $1$.

Combining the results of \cite{MTV} and \cite{T1}, we see that for
$0 \le \beta \le 1/2$ the principal values
\begin{equation}\label{pv2}
\pv \int \frac{z-w}{|z-w|^2} \,d\mu(w) =
\lim_{\varepsilon \rightarrow 0} \int_{|w-z| > \varepsilon} \frac{z-w}{|z-w|^2} \,d\mu(w)
\end{equation}
do not exist $\mu$-a.e. As we will show later, this implies that the logarithmic potential of
$\mu$ is not differentiable in the weak capacity sense at $\mu$ almost all points of $K_\beta.$

For $1/2 < \beta$ the operator $R$ with kernel $z/|z|^2$ is bounded on $L^2(\mu)$ (see \cite{MTV}) and so the principal values
\eqref{pv2} exist $\mu$ a.e. (by subsection 2.2). In this case the logarithmic potential of $\mu$ is differentiable in the ordinary
sense $\mu$-a.e., as it will be shown later.

Consider a measure function  $\Phi \colon [0,\infty) \rightarrow
[0,\infty)$, that is,  a continuous (strictly)  increasing function with $\Phi(0)=0$.
Associated with $\Phi$ there is a Cantor set $K$ whose canonical
measure satisfies $\mu(B(x,r)) \le C\,\Phi(r)$ for $x \in K$ and $0
< r < 1$, provided one has
 \begin{equation}\label{limsuph4}
 \limsup_{r \rightarrow 0} \frac{\Phi(2r)}{\Phi(r)}< 4.
 \end{equation}
The construction of the Cantor set proceeds as follows.  Define $\sigma_n$ by $ 4^{-n}= \Phi(\sigma_n)$ and then set $\lambda_n= \sigma_n /\sigma_{n-1}.$ To implement the definition of the Cantor set one needs to check that $\lambda_n < 1/2$. This follows readily for $n$ large enough from \eqref{limsuph4}.  Indeed, by \eqref{limsuph4} there exists a positive constant $C$, $C < 4$, such that $\Phi(2r) \le C\, \Phi(r)$  for $r$ sufficiently small. Thus, for $n$ large enough,
\begin{equation*}
\Phi(\sigma_{n})= \frac{\Phi(\sigma_{n-1})}{4}  \le \frac{C}{4}\,\Phi(\frac{\sigma_{n-1}} {2}) <  \Phi(\frac{\sigma_{n-1}} {2})
\end{equation*}
and so $\sigma_n <  \sigma_{n-1} /2$.


If $\sigma_n  \le  r < \sigma_{n-1}$ and $x \in K$, then $B(x,r)$ is
contained in at most $4^3$ squares $Q^n_j$.  Then  $\mu(B(x,r)) \le
4^3  \Phi(r)$.

 If the lim sup in \eqref{limsuph4} is exactly $4$ then the preceding construction fails for the function
 $\Phi(r) = r^2 \log (1/r)^{-1}$, because a measure satisfying $\mu(B(x,r)) \le C \, r^2 \log (1/r)^{-1}$,
  $x \in K$, $0 < r < 1$, is identically zero.

  The measure function giving the Cantor set $K_\beta$ is $\Phi(r)= r / \log^\beta (1/r)$.

 In dimension $d$ condition \eqref{limsuph4} should be modified replacing the upper bound $4$ by $2^d$.

\section{Proof of Theorem \ref{teo1}}

\subsection{The sufficient condition in Theorem \ref{teo1}.}
The reader will recognize in the decomposition we are going to use the basic classical argument
in \cite[p.~242]{S}. Assume that $a=0$ to simplify notation. In view of \eqref{gradpot} we set
$A_i = (d-2) \pv \int \frac{y_i}{|y|^d} \,d\mu(y)$.
We have to show that
\begin{equation*}
\lim_{r \rightarrow 0} \frac{\sup_{t > 0} t \operatorname{Cap}(\{x \in B(0,r) : Q(x) > t \})}
{\operatorname{Cap}(B(0,r))} = 0,
\end{equation*}
 where the quotient $Q(x)$ is
 \begin{equation*}
Q(x) = \frac{|u(x)-u(0) - \sum_{i=1}^d A_i x_i|}{|x|}.
\end{equation*}
 Given $r >0$ and $x \in B(0,r)$ set $\varepsilon=2|x|$. Then, denoting by $\langle v,w\rangle$ the scalar product of the
 vectors $v$ and $w$,
 \begin{equation*}
 \begin{split}
Q(x)&\le \frac{|u(x)-u(0)+(d-2)\langle R_{\varepsilon}(\mu)(0),x\rangle|}{|x|}+(d-2)\,\sup_{0<\varepsilon\le 2r} |R_{\varepsilon} (\mu)(0)-R(\mu)(0)|\\
&\equiv A_{\varepsilon}(x)+T_r.
\end{split}
 \end{equation*}
Hence
\begin{equation*}
 \begin{split}
\sup_{t>0} \frac{t\operatorname{Cap}\left\{x\in B(0,r): Q(x)>t\right\}}{\operatorname{Cap}(B(0,r))}&\le \sup_{t>0} \frac{t\operatorname{Cap}\left\{x\in B(0,r): A_{\varepsilon}(x)>\frac{t}{2}\right\}}{\operatorname{Cap}(B(0,r))}\\*[7pt]
&\quad +\sup_{t>0}\frac{t\operatorname{Cap}\left\{x\in B(0,r): T_r>\frac{t}{2}\right\}}{\operatorname{Cap}(B(0,r))}\\*[7pt]
&\le 2\sup_{t>0} \frac{t\operatorname{Cap}\left\{x\in B(0,r): A_{\varepsilon}(x)>t\right\}} {\operatorname{Cap}(B(0,r))}+2T_r.
\end{split}
\end{equation*}
Since $T_r \rightarrow 0$ as $r \rightarrow 0$ we only need to estimate the first term in the right hand side above. Clearly
\begin{equation*}
\begin{split}
A_{\varepsilon}(x)&\le \left|\int_{|y|>\varepsilon} \frac{1}{|x|} \left( \frac{1}{|x-y|^{d-2}} -\frac{1}{|y|^{d-2}} -(d-2)\left\langle \frac{y}{|y|^d},x\right\rangle\right)\,d\mu(y)\right|\\*[7pt]
&\quad +\left|\int_{|y|<\varepsilon} \frac{1}{|x|}\left(\frac{1}{|x-y|^{d-2}}-\frac{1}{|y|^{d-2}}\right)\,d\mu(y)\right|\\*[7pt]
&\equiv B_{\varepsilon} (x) +C_{\varepsilon}(x).
\end{split}
\end{equation*}
 By the mean value theorem and integration by parts, for each
positive integer $N$ one has
 \begin{equation}\label{Bmp}
\begin{split}
B_{\varepsilon}(x) &\le C\int_{|y|>\varepsilon} \frac{|x|}{|y|^d}\,d\mu(y)\\*[7pt]
&\le C \varepsilon\int^{\infty}_{\varepsilon} \frac{d\mu B(0,\rho)} {\rho^d}\\*[7pt]
&=C\varepsilon \left\{\left[ \frac{\mu B(0,\rho)}{\rho^d}\right]_{\varepsilon}^{\infty} +d \int^{\infty}_{\varepsilon} \frac{\mu B(0,\rho)}{\rho^{d+1}}\,d\rho\right\}\\*[7pt]
&\le C\varepsilon \left\{\int^{\varepsilon N}_{\varepsilon}\frac{\mu B(0,\rho)}{\rho^{d+1}}\,d\rho+\int^{\infty}_{\varepsilon N} \frac{\mu B(0,\rho)}{\rho^{d+1}}\,d\rho\right\}\\*[7pt]
&\le C\sup_{0<\rho<2rN} \frac{\mu B(0,\rho)}{\rho^{d-1}} +\frac{C}{N}\sup_{0<\rho}\frac{\mu B(0,\rho)}{\rho^{d-1}}.
\end{split}
\end{equation}
 Since $N$ is arbitrary, in view of \eqref{densitymu} we conclude that
 \begin{equation*}
\lim_{r \rightarrow 0} \left(\sup_{|x|< r} B_\varepsilon(x)\right)=0.
\end{equation*}
We now turn our attention to $C_\varepsilon(x)$. Introducing the absolute value inside the integral
 \begin{equation*}
C_{\varepsilon} (x) \le \frac{1}{|x|} \int_{|y|<\varepsilon} \frac{d\mu(y)}{|y-x|^{d-2}} +\frac{1}{|x|} \int_{|y|<\varepsilon} \frac{d\mu(y)}{|y|^{d-2}}\equiv D_{\varepsilon}(x)+ F_{\varepsilon}(x).
\end{equation*}
 The term $F_\varepsilon(x)$ is estimated readily by
\begin{equation*}
\begin{split}
F_{\varepsilon}(x)&=\frac{1}{|x|} \int^{\varepsilon}_0 \frac{d\mu B(0,\rho)} {\rho^{d-2}}\\*[7pt]
&=\frac{1}{|x|} \left\{ \left[ \frac{\mu B(0,\rho)}{\rho^{d-2}} \right]_0^{\varepsilon} +(d-2)\int_0^{\varepsilon} \frac{\mu B(0,\rho)}{\rho^{d-1}}\,d\rho\right\}\\*[7pt]
&\le C\sup_{0<\rho\le 2r} \frac{\mu B(0,\rho)}{\rho^{d-1}},
\end{split}
\end{equation*}
 and thus
\begin{equation*}
\lim_{r \rightarrow 0} \left(\sup_{|x|< r} F_\varepsilon(x)\right)=0.
\end{equation*}
It remains to bound $D_\varepsilon(x)$ and here is the only place where  a capacitary
estimate, based on \eqref{wcap}, is used. We have
\begin{equation*}
D_{\varepsilon}(x)\le \int_{|y|<2r} \frac{1}{|x-y|^{d-2}}\frac{d\mu(y)}{|y|},
\end{equation*}
and the mass of the measure $d\mu(y)/|y|$ is estimated by
\begin{equation*}
\begin{split}
\int_{|y|<2r} \frac{d\mu(y)}{|y|} &= \int^{2r}_0\frac{d\mu B(0,\rho)}{\rho}\\*[7pt]
&= \left[\frac{\mu B(0,\rho)}{\rho}\right]^{2r}_0+ \int_0^{2r}\frac{\mu B(0,\rho)}{\rho^2}\,d\rho\\*[7pt]
&\le C r^{d-2}\sup_{0<\rho<2r} \frac{\mu B(0,\rho)}{\rho^{d-1}}.
\end{split}
\end{equation*}
By \eqref{wcap}
\begin{equation*}
\sup_{t>0} \frac{t\operatorname{Cap}\left\{x\in B(0,r): D_{\varepsilon}(x)>t\right\}} {\operatorname{Cap}(B(0,r))} \le C\sup_{0<\rho<2r} \frac{\mu B(0,r)}{\rho^{d-1}},
\end{equation*}
which tends to $0$ with $r$.

\subsection{The necessary condition in Theorem \ref{teo1}.}
 The Green function for the ball $B(a,r)$ is
 $$\frac{1}{|x-a|^{d-2}} -\frac{1} {r^{d-2}}, \quad |x-a| < r.$$
 By Poisson--Green formula for $u$ and the ball~$B(a,r)$
\begin{equation*}
\begin{split}
u(a)&=\frac{1}{\sigma(\partial B(a,r))} \int_{\partial B(a,r)} (u(x)-\langle A,x-a\rangle)\,d\sigma(x)\\*[7pt]
&\quad + c_d \int_{B(a,r)} \left( \frac{1}{|x-a|^{d-2}} -\frac{1} {r^{d-2}}\right)\,d\mu(x),
\end{split}
\end{equation*}
where $A=(A_{1},\dotsc,A_d)$ is the gradient in the definition of differentiability in the capacity sense~\eqref{defdifcap} and $c_d$ is a positive constant. Since
$$
\frac{1}{|x-a|^{d-2}}-\frac{1}{r^{d-2}}\ge c_d \frac{1}{r^{d-2}},\quad |x-a|<\frac{r}{2},
$$
we obtain
\begin{equation}\label{sigmagreen}
\begin{split}
c_d \frac{1}{r^{d-2}} \mu B\left(a,\frac{r}{2}\right) &\le \frac{1}{\sigma (\partial B(a,r))}
\int_{\partial B(a,r)}|u(x)-u(a)-\langle A, x-a\rangle|\,d\sigma(x)\\*[7pt]
&=\frac{1}{\omega_{d-1} r^{d-1}} \int^{\infty}_0 \sigma\left\{x\in \partial B(a,r): |D u(x)|>t\right\}\,dt,
\end{split}
\end{equation}
where $\omega_{d-1}=\sigma (S^{d-1})$ and
\begin{equation}\label{Du}
 D u(x)=u(x)-u(a)-\langle A,x-a\rangle,\quad x\in\mathbb{R}^d.
\end{equation}
Estimating from above the potential of the measure $\chi_E(x)\,d\sigma(x)$ one readily obtains the well known estimate \cite[Corollary 5.1.14]{AH}
\begin{equation}\label{sigmacap}
c_d \,\sigma(E)\le \operatorname{Cap}(E)^{\frac{d-1}{d-2}},\quad E\subset \partial B(a,r).
\end{equation}
Hence, the right hand side of~\eqref{sigmagreen} is not greater than
\begin{equation}\label{intfron}
c_d\,\frac{1}{r^{d-1}}\int^{\infty}_0 \operatorname{Cap}^{\frac{d-1}{d-2}} \left\{x\in\partial B(a,r):|D u(x)|>t\right\}\,dt.
\end{equation}
We split the integral between $0$ and $\infty$ into two pieces: first we integrate between $0$ and $T$ and
then between $T$ and $\infty$. The positive number~$T$ will be chosen later. For the integral between $0$
and~$T$ we estimate the capacity of the set $\{x\in \partial B(a,r): |D u(x)|>t\}$ by
$\operatorname{Cap}(\partial B(a,r))=c_d r^{d-2}$. Thus
$$
\frac{c_d}{r^{d-1}} \int^T_0 \operatorname{Cap}^{\frac{d-1}{d-2}} \left\{x\in\partial B(a,r): |D u(x)|>t\right\}\,dt\le c_dT.
$$
Define $\varepsilon(r)$ as
$$
\varepsilon(r)=\frac{\sup\limits_{t>0} t\operatorname{Cap}\left\{x\in B(a,r): |D u(x)|>t\right\}} {r\operatorname{Cap}(B(a,r))},
$$
so that $\varepsilon(r)\to 0$ as $r\to 0$ if $u$ is differentiable in the capacity sense at~$a$. We get
\begin{equation*}
\begin{split}
\frac{1}{r^{d-1}}\int^{\infty}_T\operatorname{Cap}^{\frac{d-1}{d-2}} \left\{x\in\partial B(a,r):|D u(x)|>t\right\}
\,dt&\le c_d\,\left(r \varepsilon(r)\right)^{\frac{d-1}{d-2}}\int^{\infty}_T \frac{dt}{t^\frac{d-1}{d-2}}\\*[7pt]
&=c_d \,\left(r\varepsilon(r)\right)^{\frac{d-1}{d-2}}\frac{1}{T^{\frac{1}{d-2}}}.
\end{split}
\end{equation*}
The upper bound we obtain for \eqref{intfron} is
$$
c_d \left(T+\left(r\varepsilon(r)\right)\right)^{\frac{d-1}{d-2}}\frac{1}{T^{\frac{1}{d-2}}},
$$
which is minimized by $T=r\varepsilon(r)$. Therefore
$$
c_d\frac{1}{r^{d-2}} \mu B\left(a,\frac{r}{2}\right)\le r\varepsilon (r),
$$
which yields \eqref{densitymu}.


It remains to prove the existence of the principal value \eqref{pvmu}.
Assume that $a=0$ to simplify the writing. We know that there exist $A_i, \; 1 \le i \le d,$ such that
$$
\varepsilon(r)=\frac{\sup\limits_{t>0} t\operatorname{Cap}\left\{x\in B(0,r): |D u(x)|>t\right\}}
{r\operatorname{Cap}(B(0,r))}
$$
tends to $0$ with $r$. Here $Du$ is as in \eqref{Du} with $a=0$. Set
$$R_r= (d-2) \int_{|y|> r} \frac{y}{|y|^d}\,\mu(y), \quad r >0 .$$
Given $r >0$ and $x \in B(0,r), \; x \neq 0,$ we have
\begin{equation*}
\begin{split}
|\langle R_{2|x|} -A,x\rangle| &\le Du(x)
+|u(x)-u(0)-\langle R_{2|x|},x\rangle|\\
&\equiv Du(x)+ Eu(x),
\end{split}
\end{equation*}
and
\begin{equation*}
\begin{split}
\eta(r)&:= \sup_{t>0} \frac{t\operatorname{Cap}\left\{x\in B(0,r):|\langle R_{2|x|}-A,x\rangle|>t\right\}}{r\operatorname{Cap}(B(0,r))}\\*[7pt]
&\le 2\varepsilon (r)+2 \sup_{t>0}\frac{t\operatorname{Cap}\left\{x\in B(0,r):Eu(x)>t\right\}}{r\operatorname{Cap}(B(0,r))}.
\end{split}
\end{equation*}
In the proof the sufficiency in subsection 3.1 we showed that the second term in the right hand side of the
preceding inequality tends to $0$ with $r$. Therefore $\eta(r)$ tends to $0$ as $r$ tends to $0$.

If $R_r \neq A$ define
\begin{equation}\label{kar}
K_r=\left\{x\in\mathbb{R}^d: |x|=\frac{r}{2}\quad\text{and}\quad \left\langle
\frac{x}{|x|},\frac{R_r-A}{|R_r-A|}\right\rangle\ge \frac{1}{\sqrt{2}}\right\}.
\end{equation}
Observe that $K_r$ is the intersection of the sphere of center $0$ and radius $r/2$ with a cone with vertex
at $0$, axis determined by the unit vector in the direction of $R_r - A$, and aperture $\pi/4$. A dilation
argument shows that $\Capa(K_r)= c_d\, r^{d-2}$. Hence, if $R_r \neq A$,
\begin{equation*}
\operatorname{Cap}\left\{ x\in B(0,r): |\langle R_{2|x|}-A,x\rangle|> |R_r-A|\frac{r}{2^{5/2}}\right\}\ge
\operatorname{Cap}(K_r)=c_d r^{d-2}.
\end{equation*}
Taking $t=|R_r -A| r/2^{5/2}$ in the definition of $\eta(r)$ we get
\begin{equation*}
\eta(r)\ge c|R_r-A|,
\end{equation*}
and therefore
$$
\lim_{r \rightarrow 0} R_r = A.
$$

\section{Proof of Theorem \ref{teo2}}
Let $\mu$ be a finite Borel measure in $\Rd, \, d \ge 3,$ and let
$u$ be its Riesz potential, as in \eqref{pot}. In proving Theorem
\ref{teo2} we can assume, without loss of generality, that $\mu$ is
positive. Let $E$ be
 the set of points where $u$
 is not differentiable in the capacity sense. Take a positive finite Borel measure~$m$ supported on a compact
subset of $E$ satisfying  $m(B(x,r)) \le r^{d-1},$ $ \; x \in \Rd, \; 0< r,$ $\lim_{r \rightarrow 0} m(B(x,r)
)/ r^{d-1}=0, \;x \in \Rd,$ and such that the operator $R$ with kernel $x-y/|x-y|^d$ is bounded on $L^2(m)$. We will show that $u$ is
differentiable in the capacity sense $m$ a.e. Hence $m$ must be identically zero and thus $k^c(E)=0$,
as desired.

The Radon--Nikodym decomposition of $\mu$ with respect to $m$ is $\mu = f \,m + \mu_s$
where $f \in L^1(m)$ and $\mu_s$ is singular with respect to $m$. On the one hand
one has
$$
\lim_{r \rightarrow 0} \frac{\mu_s(B(a,r))}{m(B(a,r))}=0
$$
at $m$ almost all points $a$ and, on the other hand, $m$ almost all points are Lebesgue points of $f$.
Hence
$$
\mu(B(a,r)) \le C(a) \, m(B(a,r)), \quad 0 < r,
$$
at $m$ almost all points $a$, $C(a)$ being a constant which depends only on the point $a$. Since
the operator $R$ with kernel $x-y/|x-y|^d$ is bounded on $L^2(m)$, by subsection 2.2 the principal value
\begin{equation*}
\pv \int \frac{a-y}{|a-y|^d}\, d\mu(y)
\end{equation*}
exists at $m$ a.e. Hence we can apply the sufficient condition in Theorem~\ref{teo1} to
conclude that  the potential $u$ of $\mu$ is differentiable in the capacity sense at $m$ almost all points,
which completes the proof.

\begin{example} Let $\sigma$ be the surface measure on the unit sphere $S=\{x \in \Rd : |x| =1\}$. Since $\sigma$ has non-zero
$(d-1)$-dimensional density, one can apply Theorem \ref{teo1} to conclude that the Riesz potential
$1/|x|^{d-2}*\sigma, \; d \ge 3$  is not differentiable in the capacity sense
at any point of $S.$ One can avoid appealing to Theorem \ref{teo1} and  make a direct
calculation, which works also in dimension $d=2$ for the
logarithmic potential $\log (1 / |z|)* \sigma.$ Since $S$ has positive and finite
$(d-1)$-dimensional Hausdorff measure we get a
satisfactory example showing that Theorem \ref{teo2} is sharp.
\end{example}

\section{Proof and sharpness of Theorem \ref{teo3}}
Let $\mu$ be a finite Borel measure and $u$ its logarithmic
potential, as in  \eqref{potlog}. For the purpose of proving Theorem
\ref{teo3} one can assume, without loss of generality that $\mu$ is
positive. Let $E$ stand for the set of points at which $u$ is not
differentiable in the weak capacity sense. Take a positive finite
Borel measure $m$ with compact support contained in $E$ satisfying
the growth condition $m(B(z,r) \le \varphi(r), \; 0 <r,$ where
$\varphi$ is the function in \eqref{fi}. If we see that $u$ is
differentiable in the weak capacity sense at $m$ almost all points,
then $m$ has to be identically $0$ and hence $H^\varphi(E)=0$.

The Radon--Nikodym decomposition of $\mu$ with respect to $m$ has the form  $\mu = f m+ \mu_s$, with
$f \in L^1(m)$ and $\mu_s$ singular with respect to $m$. Given a point $a$ set $\nu = (f-f(a)) m + \mu_s$
so that $\mu = \nu + f(a)m$. At $m$ almost all points $a$ one has
\begin{equation}\label{nug}
 |\nu|(B(a,r)) \le \eta(r)\,\varphi(r)
\end{equation}
where $\eta$ is a function depending on $a$ with  $\eta(r) \rightarrow 0$ as $r \rightarrow 0$. We plan to
show that the logarithmic potential
of $\nu$ is differentiable in the weak capacity sense at the
point $a$ if \eqref{nug} is satisfied and the principal value $\pv (a-w)/|a-w|^2 \,d\nu(w)$ exists. As we
mentioned in subsection 2.4 the growth condition fulfilled by
 $m$ implies that the operator $R$ with kernel $(z-w)/|z-w|^2$ is bounded on $L^2(m)$, which yields $m$
a.e.\ existence of the principal values $\pv (a-w)/|a-w|^2 \,d\nu(w)$  for each finite Borel measure $\nu$
(by subsection 2.2). Finally we will show that the logarithmic potential of $m$ is
differentiable in  the ordinary sense  $m$ a.e. This will complete the proof of Theorem \ref{teo3}.

We first deal with the logarithmic potential of $m$.

\begin{lemma}\label{lemma1}
Let $m$ be a positive finite Borel measure such that
\begin{equation*}
 m(B(z,r)) \le \eta(r)\, r, \quad z \in \C, \quad 0 <r,
\end{equation*}
with $\eta(r) \rightarrow 0$  as $r \rightarrow 0$,  and the principal value
\begin{equation*}
\pv \int \frac{a-w}{|a-w|^2} \,d\mu(w)
\end{equation*}
exists at the point $a$. Then the logarithmic potential of $m$ is differentiable in the ordinary
sense at the point $a$.
\end{lemma}

\begin{proof}
 Assume that $a=0$ and set
 \begin{equation*}\label{Aplane}
A = \pv \int \frac{w}{|w|^2} \,dm(w),
\end{equation*}
 \begin{equation*}\label{Rep}
R_\varepsilon = \int_{|w|> \varepsilon} \frac{w}{|w|^2} \,dm(w), \quad \varepsilon >0,
\end{equation*}
 \begin{equation*}\label{Qplane}
Q(z) = \frac{|u(z)-u(0) - \langle A, z \rangle |}{|z|}, \quad z \in \C \setminus\{0\}.
\end{equation*}
 Then
\begin{equation*}\label{Qzmp}
Q(z) \le  \frac{|u(z)-u(0) - \langle R_{2|z|}, z \rangle |}{|z|}+ |R_{2|z|} - A|,
\quad z \in \C \setminus\{0\}.
\end{equation*}
 The second term in the right hand side above tends to $0$ with $z$ and the first can be estimated by
 \begin{equation*}\label{fet}
\begin{split}
\frac{1}{|z|} &\left| \int_{|w|>2|z|} \left( \log \frac{1}{|w-z|}-\log \frac{1}{|w|}-\left\langle\frac{w}{|w|^2},z\right\rangle\right)\right|\,dm(w)\\*[7pt]
&\quad+\frac{1}{|z|}  \int_{|w|<2|z|} \left( \log \frac{3|z|}{|w-z|}-\log \frac{3|z|}{|w|}\right)\,dm(w)\\*[7pt]
&\equiv A(z)+B(z).
\end{split}
\end{equation*}
The term $A(z)$ is treated by the mean value theorem and integration by parts similarly to what was done
in the proof of the sufficiency for Theorem \ref{teo1}. One gets
 \begin{equation*}\label{Azmp}
A(z)\le C |z|+C\eta (N|z|)+C \, \frac{1}{N} \,\sup_{0<\rho<\frac{1}{4}} \eta(\rho),
\end{equation*}
where $N$ is an arbitrary positive integer. Thus $\lim_{z \rightarrow 0} A(z)=0$.
We estimate $B(z)$ by
 \begin{equation*}\label{Bzmp}
\begin{split}
B(z)&\le \frac{1}{|z|} \int_{|w-z|<3|z|}\log \frac{3|z|} {|w-z|} \,dm(w)+ \frac{1}{|z|} \int_{|w|<3 |z|}\log\frac{3|z|}{|w|}\,dm(w)\\*[7pt]
&\equiv C(z)+D(z).
\end{split}
\end{equation*}
For $C(z)$ one has
 \begin{equation*}\label{Czmp}
\begin{split}
C(z)&=\frac{1}{|z|}\int^{3|z|}_0\log\frac{3|z|}{\rho} \,dm B(z,\rho)\\*[7pt]
&=\frac{1}{|z|}\int_0^{3|z|}\frac{mB(z,\rho)}{\rho}\,d\rho\\*[7pt]
&\le 3\sup_{0<\rho<3|z|}\eta( \rho),
\end{split}
\end{equation*}
which yields $\lim_{z \rightarrow 0} C(z)=0$.
A similar estimate for $D(z)$ gives that $\lim_{z \rightarrow 0} D(z)=0$, which completes the proof.
 \end{proof}

 It remains to deal with the differentiability in the weak capacity sense of the logarithmic potential of $\nu$. We can
 assume without loss of generality that $\nu$ is a positive measure. The following lemma settles the question.

\begin{lemma}\label{lemma2}
Let $\nu$ be a positive finite Borel measure such that
\begin{equation*}
 \lim_{r \rightarrow 0} \frac{\nu(B(a,r))}{\varphi(r)}=0
\end{equation*}
and the principal value
\begin{equation}\label{pvm}
\pv \int \frac{a-w}{|a-w|^2} \,d\nu(w)
\end{equation}
exists. Then the logarithmic potential of $\nu$ is differentiable in the weak capacity
sense at the point $a$.
\end{lemma}
\begin{proof}
 Assume that $a=0$ and set
 \begin{equation*}\label{Anu}
A = \pv \int \frac{w}{|w|^2} \,d\nu(w),
\end{equation*}
 \begin{equation*}\label{Dupla}
Du(z) = |u(z)-u(0) - \langle A, z \rangle |,
\end{equation*}
 \begin{equation*}\label{Repnu}
R_\varepsilon = \int_{|w|> \varepsilon} \frac{w}{|w|^2} \,d\nu(w), \quad \varepsilon >0,
\end{equation*}
 \begin{equation*}\label{Eupla}
Eu(z) = |u(z)-u(0) - \langle R_{2|z|} , z \rangle |.
\end{equation*}
 Then
\begin{equation*}
\begin{split}
\sup_{t>0} \frac{t\operatorname{Cap}\left\{z\in B(0,r): Du(z)>t\right\}} {r\operatorname{Cap}(B(0,r))}
&\le 2\sup_{t>0} \frac{t\operatorname{Cap}
\left\{z\in B(0,r): Eu(z)>t\right\}}{r\operatorname{Cap}(B(0,r))}\\*[7pt]
&\quad +2\sup_{|z|<r} |R_{2|z|}-A|.
\end{split}
\end{equation*}
 The second term in the right hand side above tends to $0$ as $r \rightarrow 0$.
 To estimate the first note that for $z \in B(0,r)$
 \begin{equation}\label{Eump}
Eu(z)\le Cr\alpha(r)+\int_{|w|<2r}\log \frac{1}{|w-z|}\,d\nu(w) +\int_{|w|<2r}\log \frac{1}{|w|}\,d\nu(w),
\end{equation}
 with $\alpha(r) \rightarrow 0$ as $r \rightarrow 0$. This is proved as   in
 the sufficiency part of Theorem \ref{teo1}.

 The third term in the right hand side of \eqref{Eump} is
 \begin{equation}\label{3t}
\int_0^{2r}\log\frac{1}{\rho}\,d\nu B(0,\rho)=(\log\frac{1}{2r}) \nu B(0,2r) +\int_0^{2r}\frac{\nu B(0,\rho)}{\rho}\,d\rho \le
Cr\sup_{0<\rho<2r}\eta(\rho),
\end{equation}
where $\eta(\rho) = \nu (B(0,\rho) / \varphi(\rho)$.
 The second term in the right hand side of \eqref{Eump} is the logarithmic potential $P(\chi_{B(0,2r)} \nu)$
 of the measure $\chi_{B(0,2r)} \nu.$  This is estimated via \eqref{wcap} and we obtain
\begin{equation*}
\begin{split}
\sup_{t>0}\frac{t\operatorname{Cap}\left\{z\in B(0,r): P(\chi_{B(0,2r)}\nu)>t\right\}}
{r\operatorname{Cap}(B(0,r))}
&\le C\log (\frac{1}{r}) \frac{\nu (B(0,2r))}{r}\\*[7pt]
&\le C\eta(2r).
\end{split}
\end{equation*}
 Therefore gathering all previous inequalities
 \begin{equation*}
\sup_{t>0}\frac{t\operatorname{Cap}\left\{z\in B(0,r): Eu(z)>t\right\}}
{r\operatorname{Cap}(B(0,r))}\le C\alpha(r)+\sup_{0<\rho<2r}\eta(\rho),
\end{equation*}
 which tends to $0$ with $r$.
\end{proof}

There are a couple of necessary conditions for differentiability in the weak capacity sense which provide
interesting examples of positive measures with non-differentiable logarithmic potentials. The first is the complete
analogue of the
necessary condition in Theorem \ref{teo2} concerning the vanishing of $(d-1)$-dimensional density.

\begin{lemma}\label{lemma3}
Let $\mu$ be a positive finite Borel measure such that its logarithmic potential is differentiable in the weak
capacity sense at the point $a \in \C.$ Then
\begin{equation}\label{zerodensitymu2}
\lim_{r \rightarrow 0} \frac{\mu(B(a,r))}{r} =0.
\end{equation}
\end{lemma}

Then the logarithmic potential of the arc length measure on $S=\{z \in \C : |z|=1 \}$ is not differentiable
in the
weak capacity sense at any point of $S$. In the same vein, the logarithmic potential of the length measure
on the corner quarters Cantor set $K_0$ is not differentiable in the
weak capacity sense at any point of $K_0$.

\begin{proof}[Proof of Lemma \ref{lemma3}]
The argument presented for the necessary condition in Theorem \ref{teo1} works perfectly well in dimension
$d=2$. Indeed, one can replace \eqref{sigmacap} by
$$
c\,\sigma (E)\le \exp\left(-\frac{1}{\operatorname{Cap}(E)}\right),\quad E\subset \partial B(a,r),
$$
where $c$ stands for a small positive constant, and argue similarly (\cite[Corollary 5.1.14]{AH}).  There is, however,
an alternative argument which goes as follows. Using the notation introduced in the proof of the necessary condition in
Theorem \ref{teo1} and recalling that
the Green function of the disc of center $a$ and radius $r$ is $\log (r/|z-a|)$ one gets
$$
c \,\mu B\left(a,\frac{r}{2}\right)\le \frac{1}{2\pi r} \int_{\partial B(a,r)} |D u(x)|\,d\sigma(x).
$$
See \eqref{sigmagreen} and \eqref{Du} .
Then, for at least one point $p=p(r) \in \partial B(a,r)$, we have, for a smaller constant $c$,
$$
c\,\mu  B\left(a,\frac{r}{2}\right)\le |D u(p)|.
$$
We claim that
\begin{equation}\label{capgran}
\operatorname{Cap}\left \{x\in B(a,r): |D u(x)|> c\,\mu B\left(a,\frac{r}{4}\right)\right\}\ge
c\,\operatorname{Cap} (B(a,r)).
\end{equation}
Let us finish the argument assuming \eqref{capgran}. Taking $t=c\,\mu B (a,\frac{r}{4})$ we get
\begin{equation*}
\begin{split}
\varepsilon (r)&\ge \frac{c\,\mu B\left(a,\frac{r}{4}\right)\operatorname{Cap}\left\{x\in B(a,r): |D u(x)|>c\,\mu B\left(a,\frac{r}{4}\right)\right \}}{r\operatorname{Cap}B(a,r)}\\*[7pt]
&\ge c\,\frac{\mu B\left(a,\frac{r}{4}\right)} {r},
\end{split}
\end{equation*}
which gives \eqref{zerodensitymu2}.

To show the claim take $\rho$, $\frac{r}{2}<\rho<r$. Then there exists $p=p(\rho)$ with $|p-a|=\rho$ and
$$
|D u(p)|>c\,\mu B\left(a,\frac{\rho}{2}\right)\ge c\,\mu B\left(a,\frac{r}{4}\right).
$$
The mapping
$$
p\longrightarrow \phi(p)=|p-a|
$$
is Lipschitz with constant~$1$ and
$$
\phi \left\{x\in B(a,r):|D u(x)|>c\, \mu B\left(a,\frac{r}{4}\right)\right\} \supseteq \left[\frac{r}{2},r\right].
$$
Since Lipschitz mappings with constant~$1$ do not increase capacity we conclude that
$$
\operatorname{Cap}\left\{x\in B(a,r):|D u(x)|>c\,\mu \left(a,\frac{r}{4}\right)\right\}\ge
\operatorname{Cap}\left(\left[\frac{r}{2},r\right]\right) \simeq c\,\operatorname{Cap} B(a,r). \rlap{\hspace*{.56cm}\qed}
$$
\renewcommand{\qedsymbol}{}
\end{proof}

We do not know if the existence of principal values is a necessary condition for differentiability in the
capacity sense in dimension $d=2$. We can prove, however, the following.

\begin{lemma}\label{lemma4}
Let $\mu$ be a positive finite Borel measure such that its logarithmic potential is differentiable in the weak
capacity sense at the point $a \in \C.$ Assume also that one of the following two conditions is satisfied
\begin{enumerate}
\item[(i)]\begin{equation*}\label{uniformzerodensity}
\mu(B(z,r)) \le C\, \eta(r)\,r,  \quad z \in \C, \quad 0<r,
\end{equation*}
with $\eta(r) \rightarrow 0$ as $r \rightarrow 0$.
\item[(ii)] \begin{equation*}\label{densityfi}
\lim_{r \rightarrow 0} \frac{\mu(B(a,r))}{\varphi (r)} =0.
\end{equation*}
 \end{enumerate}
Then the principal value
\begin{equation}\label{pvmu3}
\pv \int \frac{a-w}{|a-w|^2} \,d\mu(w)
\end{equation}
exists.
\end{lemma}


Let $K$ be a Cantor set satisfying the condition
$$
\sum_{n=1}^\infty \frac{1}{(4^n \sigma_n)^2} = \infty.
$$
In the scale of the Cantor sets $K_\beta$ this is equivalent to $0 \le \beta  \le 1/2$.
Then the operator~$R$ with kernel $(z-w)/|z-w|^2$ is unbounded on $L^2(\mu)$, where $\mu$ is the canonical
measure on~$K$ (see \cite{MTV}) and the principal
value \eqref{pvmu3} does not exist for $\mu$ almost all points $a \in K$  (see \cite{T1}). Hence, by Lemma \ref{lemma4} the
logarithmic potential of $\mu$ is not differentiable in the weak capacity sense at $\mu$ almost all
points $a \in K$.

\begin{proof}[Proof of Lemma \ref{lemma4}]
The proof parallels that of the necessary condition in Theorem \ref{teo1}. If $u$ is the logarithmic
potential of $\mu$ and $a=0$, then one proves there that, setting
$$
Eu(z)= |u(z)- u(0) - \langle R_{2|z|}, z \rangle |,
$$
one has
\begin{equation*}
\lim_{r \rightarrow 0} \frac{\sup\limits_{t>0} t\operatorname{Cap}\left\{z\in B(0,r): |E u(z)|>t\right\}}
{r\operatorname{Cap}(B(0,r))}=0.
\end{equation*}
This is proven in Lemma \ref{lemma1} under the assumption~(i) and in Lemma \ref{lemma2} under the assumption~(ii). The rest of the proof is the same, except for the fact that now the set $K_r$ of \eqref{kar} satisfies $\Capa(K_r) \simeq 1/\log(1/r)$.
\end{proof}

Theorem \ref{teo3} is sharp in the scale of Hausdorff measures. This is the content of the following result.

\begin{theorem}\label{teo6}
 Let $\Phi \colon[0,\infty) \rightarrow [0,\infty)$, $\Phi(0)=0$, be a continuous (strictly)  increasing function such that
 \begin{equation}\label{limsuph}
 \limsup_{t \rightarrow 0} \frac{\Phi(2r)}{\Phi(r)}< 4,
 \end{equation}
 and
 \begin{equation*}
  M(r): = \frac{\Phi(r)}{\varphi(r)} \rightarrow \infty, \quad \text{as} \quad r \rightarrow \infty.
 \end{equation*}
Then there exists a compact set $K$ with $H^\Phi(K) >0$ and a finite Borel measure whose logarithmic potential
is not differentiable in the weak capacity sense at $H^\Phi$ almost all points of $K$.
\end{theorem}

This means that you cannot get any condition better than $H^\varphi(E)=0$ on the set $E$ of points
of non differentiability in the weak capacity sense for the logarithmic potential of a finite Borel measure.
In particular, there exists a  finite Borel measure whose logarithmic potential is not differentiable in
the weak capacity sense on a set of positive  $C^1$ harmonic capacity (that is,
 positive continuous analytic capacity). Hence the size of the exceptional sets may be larger in dimension $2$
 than in higher dimensions.  See subsection 2.5 for a comment on condition \eqref{limsuph}.
 
 
\begin{proof}[Proof of Theorem \ref{teo6}]
Let $K$ be the Cantor set associated with $\Phi$ and let $\mu$ be
its canonical measure (see subsection 2.5).  We aim at constructing a finite Borel
measure $\nu$ whose logarithmic potential is not differentiable in
the weak capacity sense at $\mu$ almost all points.

Given a positive integer $n$ take another large positive integer
$N_n$ to be determined later. Given a square $Q^{N_n}_j$ of generation $N_n$, $1 \le j \le 4^{N_n}$,
choose a square of generation $N_n+n$ (it is not important which one
is chosen). Denote the center of the chosen square of
generation $N_n+n$ contained in $Q^{N_n}_j$ by $p^n_j, \; 1 \le j \le 4^{N_n}$. Set
$$
E_n = \bigcup_{j=1}^{4^{N_n}} B(p^n_j, \sqrt{2}\,\sigma_{N_n+n}),
$$
$$
D_m =\bigcup_{n=m}^{\infty} E_n,
$$
and
$$
D= \bigcap_{m=1}^\infty D_m.
$$
Clearly $\mu(E_n)= 4^{N_n} 4^{-(N_n+n)}=4^{-n}$, and this is the only reason why
we have descended $n$ more generations after $N_n$. Hence
$\mu(D_m) \le \sum_{n=m}^\infty 4^{-n}$ and $\mu(D)=0.$

Take $a \in K \setminus D.$ Then $a \neq p^n_j, \; $ for all $n$ and
$j$, because $p^n_j \notin K.$ Since $a \notin D$, $a \notin D_m$
for some $m$, and so $a \notin B(p^n_j, \sqrt{2}\,\sigma_{N_n+n})$
for all $n \ge m$ and all $j.$

We proceed now to define the finite Borel measure whose logarithmic
potential is not differentiable in the weak capacity sense at all
points $a \notin D.$ First note that if $B$ is the unit disc
$B(0,1)$ we have
\begin{equation}\label{L}
L(z) : = \chi_B(z) \,\log \frac{1}{|z|} = \log \frac{1}{|z|}*
\left(\delta_0 - \frac{d \sigma}{2 \pi} \right)
\end{equation}
where $\delta_0 $ is the Dirac delta at the origin and $d \sigma$
the arc-length measure on the unit circle $\{z : |z|=1\}$. The second identity in \eqref{L} can
be shown by computing the Laplacian of $L$ and recalling that $1/(2
\pi) \log |z|$ is the fundamental solution of the Laplacian in the
plane. Translating and dilating we get
\begin{equation*}
L(\frac{1}{\rho} (z-p)) =   \log \frac{1}{|z|}* \left(\delta_{p} -
\frac{d \sigma_{p,\rho}}{2 \pi \rho} \right), \quad p \in \C, \quad
0 < \rho,
\end{equation*}
where $\delta_{p}$ is the Dirac delta at the point $p$ and $d
\sigma_{p,\rho}$ is arc-length measure on $\partial B(p,\rho).$
Define
\begin{equation*}
\nu = \sum_{n=1}^\infty \frac{1}{n^2 \,4^{N_n}} \sum_{j=1}^{4^{N_n}}
\left(\delta_{p^n_j} - \frac{d \sigma_{p^n_j,\sqrt{2}\,\,
\sigma_{N_n+n}}}{2 \pi \sqrt{2}\, \sigma_{N_n+n}} \right)
\end{equation*}
which is a finite Borel measure because $\|\nu\| \le
\sum_{n=1}^\infty 2/ n^2.$ The logarithmic potential of $\nu$ is
\begin{equation*}
u(z) = \sum_{n=1}^\infty \frac{1}{n^2 \,4^{N_n}}
\sum_{j=1}^{4^{N_n}} L(\frac{1}{\sqrt{2}\, \sigma_{N_n+n}}
(z-p^n_j)).
\end{equation*}
To simplify notation write
\begin{equation*}
S_n(z) = \frac{1}{4^{N_n}} \sum_{j=1}^{4^{N_n}}
L(\frac{1}{\sqrt{2}\, \sigma_{N_n+n}} (z-p^n_j)).
\end{equation*}
Given $a \in K \setminus D$ as before,  we have $a \notin B(p^n_j,
\sqrt{2}\,\sigma_{N_n+n})$ for all $n \ge m$ and all $j.$ Thus
$S_n(a)=0$, for all $n \ge m$ and consequently
\begin{equation*}
u(z)-u(a) = \sum_{n=1}^{m-1} \frac{1}{n^2}
\left(S_n(z)-S_n(a)\right) + \sum_{n=m}^{\infty} \frac{1}{n^2}
S_n(z).
\end{equation*}
Recall that  $a \neq p^n_j,$ for all $n$ and $j$. If $r>0 $ is
small enough, then $p^n_j \notin B(a,r)$, for $ n \le m-1$ and all
$j.$ Therefore
$$
 \sum_{n=1}^{m-1} \frac{1}{n^2}
\left(S_n(z)-S_n(a)\right)
$$
is smooth on $B(a,r).$  Consequently the differentiability
properties of $u$ at the point $a$ depend only on
$$
R(z) : = \sum_{n=m}^{\infty} \frac{1}{n^2} S_n(z).
$$
 Assume that $R$ is differentiable in the weak capacity
sense at $a$. It is a general fact that then $R$ is Lipschitz in the
weak capacity sense at the point $a$. Since $R(a)=0$ this means that
\begin{equation*}
\frac{\sup_{t > 0} t \operatorname{Cap}(\{z \in B(a,r) : |R(z)|
> t \})}{r \,\operatorname{Cap}(B(a,r))} \le C_a, \quad 0 < r < 1/4,
\end{equation*}
for some constant $C_a$ depending only on $a.$  To disprove the
preceding inequality we take radii of the form
\begin{equation*}
r=r_k= \sqrt{2}\,\sigma_{N_k},\quad k =1, 2, \dots
\end{equation*}
For each $k$ the point $a$ belongs to a square $Q^{N_k}_j$ of
generation $N_k$. Hence
$$
B(p^k_j, \sqrt{2}\,\sigma_{N_k+k}) \subset B(a,
\sqrt{2}\,\sigma_{N_k} ).
$$
 Take $k \ge m$ large enough so that $p^n_j \notin B(a,\sqrt{2}\,\sigma_{N_k})$, for $ n
\le m-1$ and all $j.$

Then
\begin{equation*}
R(z) \ge \frac{1}{k^2} S_k(z) \ge \frac{1}{k^2
4^{N_k}}\,L(\frac{1}{\sqrt{2}\, \sigma_{N_k+k}} (z-p^k_j)).
\end{equation*}

\pagebreak

\noindent
The right hand side of the inequality above is larger than $t$ if
and only if
$$|z-p^k_j| < \sqrt{2}\, \sigma_{N_k+k} \;e^{-k^2  4^{N_k}
t}.$$
Thus
\begin{equation*}\label{capR}
\begin{split}
\sup\limits_{t>0} t\operatorname{Cap}\left\{z\in B(a,\sqrt{2}\,
\sigma_{N_k}): |R(z)|>t \right\} & \ge c\, \sup\limits_{t>0}
\frac{t}{k^2 4^{N_k} t + \log \frac{1}{\sqrt{2}\, \sigma_{N_k+k}}}\\
& = c\,\frac{1}{k^2 4^{N_k}}
\end{split}
\end{equation*}
and
\begin{equation*}\label{Rlipcap}
\begin{split}
\sup\limits_{1/4 > r>0} \frac{\sup\limits_{t>0}
t\operatorname{Cap}\left\{z\in B(a,r): |R(z)|>t \right\}}{r\,
\Capa(B(a,r))} & \ge c\,\frac{1}{k^2 4^{N_k} \sigma_{N_k}} \log
\frac{1}{\sigma_{N_k}}\\*[7pt]
& = c\,\frac{M(\sigma_{N_k})}{k^2}.
\end{split}
\end{equation*}
Given $k$ take now $N_k$ so that $M(\sigma_{N_k}) \ge k^3,$  which
is possible because $M(r)\rightarrow \infty$ as $r \rightarrow 0.$
\end{proof}

\section{Proof and sharpness of Theorem \ref{teo4}}

 The proof follows the pattern
of that of Theorem \ref{teo3}. Following the details of the argument
below should provide a clear explanation of the role of the function
$\psi$ in \eqref{psi} as a substitute for the function $\varphi$ in
Theorem \ref{teo3}.

Let $\mu$ be a finite Borel measure and $u$ its logarithmic
potential. We assume, without loss of generality, that $\mu$ is positive. Let $E$ stand for the set of points at which $u$ is not
differentiable in the capacity sense. Take a positive finite Borel measure
$m$ with compact support contained in $E$ satisfying the growth
condition $m(B(z,r) \le \psi(r), \;z \in \C, \; 0 <r.$ If we see
that $u$ is differentiable in the capacity sense at~$m$ almost all
points, then $m$ has to be identically $0$ and hence $H^\psi(E)=0$.

The Radon--Nikodym decomposition of $\mu$ with respect to $m$ has the
form  $\mu = f m+ \mu_s$, with $f \in L^1(m)$ and $\mu_s$ singular
with respect to $m$. Given a point $a$ set $\nu = (f-f(a)) m +
\mu_s$ so that $\mu = \nu + f(a)m$. At $m$ almost all points $a$ one
has
\begin{equation}\label{nug2}
 |\nu|(B(a,r)) \le \eta_a(r)\,\psi(r)
\end{equation}
where $\eta_a$ is a function, possibly depending on $a$, with
$\eta_a(r) \rightarrow 0$ as $r \rightarrow 0$. We plan to show that
the logarithmic potential of $\nu$ is differentiable in the capacity
sense at the point $a$ if \eqref{nug2} holds. This will complete the proof
because the logarithmic
potential of the measure $m$ is of class $C^1(\C).$  This is a
consequence of the fact that its gradient $- 1/2 \int
(\overline{w}-\overline{z})^{-1}\,dm(w)$ is a continuous function,
which in turn follows from the uniform growth condition $m(B(z,r)
\le \psi(r), \; z \in \C, \; 0 <r.$

\pagebreak

Let us proceed to prove that the logarithmic potential of $\nu$ is
differentiable in the capacity sense at the point $a$ if
\eqref{nug2} holds. If $ |\nu|(B(a,r)) \le C_a\,\psi(r)$ holds
for $0 < r$
with a constant $C_a$, which may  depend on $a$,  then it is easily seen that
$\int |w-a|^{-1}\,d|\nu|(w) < \infty$. Hence the principal value
$\pv (a-w)/|a-w|^2 \,d\nu(w)$ exists.
Without loss of generality we can assume $\nu$
to be a positive measure.  Assume that $a=0$ and set
 \begin{equation*}\label{Anu2}
A = \int \frac{w}{|w|^2} \,d\nu(w),
\end{equation*}
 \begin{equation*}\label{Qpla}
Qu(z) = \frac{|u(z)-u(0) - \langle A, z \rangle |}{|z|}, \quad z
\neq 0,
\end{equation*}
 \begin{equation*}\label{Repnu2}
R_\varepsilon = \int_{|w|> \varepsilon} \frac{w}{|w|^2} \,d\nu(w), \quad
\varepsilon >0,
\end{equation*}
 \begin{equation*}\label{Eupla}
Eu(z) = \frac{|u(z)-u(0) - \langle R_{2|z|} , z \rangle |}{|z|},
\quad z \neq 0.
\end{equation*}
 Then
\begin{equation*}
\begin{split}
\sup_{t>0} \frac{t\operatorname{Cap}\left\{ z\in B(0,r): Qu(z)>t\right\}}
{\operatorname{Cap}(B(0,r))}
&\le \sup_{t>0}\frac{t\operatorname{Cap}\left\{z\in B(0,r): Eu(z)>t\right\}}
{\operatorname{Cap}(B(0,r))}\\*[3pt]
&\quad +\sup_{|z|<r}|R-R_{2|z|}|.
\end{split}
\end{equation*}
 The second term in the right hand side above tends to $0$ as $r \rightarrow 0$.
 To estimate the first one notes that
 \begin{equation*}
\begin{split}
Eu(z)&\le C|z| \int_{|w|>2|z|} \frac{d\nu(w)}{|w|^2}+\frac{1}{|z|} \int_{|w|<2|z|}\log
\frac{1}{|w-z|}\,d\nu (w)\\*[3pt]
&\quad +\frac{1}{|z|}\int_{|w|<2|z|} \log\frac{1}{|w|}\,d\nu(w)\\*[3pt]
&\equiv A(z)+B(z)+C(z).
\end{split}
\end{equation*}
Set $\eta = \eta_a$  for the sake of notational simplicity.
Integrating by parts we get for all positive integers $N$
\begin{equation*}
\begin{split}
A(z)&\le C|z|\left[ \frac{\nu B(0,\rho)}{\rho}\right]^{1/4}_{2|z|} + C|z|
\int^{1/4}_{2|z|}\frac{\nu B (0,\rho)}{\rho^3}\,d\rho\\*[3pt]
&\le C|z| \|\nu\|+C|z|\int^{2|z|N}_{2|z|} \frac{\eta(\rho)}{\rho^2\log^2(\rho)}\,d\rho\\*[3pt]
&\quad +C|z|\int^{1/4}_{2|z|N} \frac{\eta(\rho)}{\rho^2\log^2(\rho)}\,d\rho\\*[3pt]
&\le C|z| \|\nu\| + C\frac{\|\eta\|_{\infty}}{\log^2 (2|z|N)}+ C\,\frac{\|\eta\|_{\infty}}{N}.
\end{split}
\end{equation*}
Since $N$ is arbitrary we see that
$$
\lim_{r \rightarrow 0}\, \sup\limits_{|z|< r} \,A(z) = 0.
$$
The term $C(z)$ is estimated similarly via an integration by parts.
We obtain
\begin{equation*}
\begin{split}
C(z)&=\frac{1}{|z|} \left[\log (\frac{1}{\rho}) \nu B (0,\rho)\right]_0^{2|z|}+\frac{1}{|z|}\int_0^{2|z|}\frac{\nu B(0,\rho)}{\rho}\,d\rho\\*[7pt]
&\le \frac{\|\nu\|_{\infty}}{\log \frac{1}{2|z|}}+\frac{1}{|z|} \int_0^{2|z|} \frac{\eta(\rho)}{\log^2(\rho)}\,d\rho\\*[7pt]
&\le \frac{\|\eta\|_{\infty}}{\log\frac{1}{2|z|}}+\frac{\|\eta\|_{\infty}}{\log^2(2|z|)},
\end{split}
\end{equation*}
and so
$$
\lim_{r \rightarrow 0}\, \sup\limits_{|z|< r} \,C(z) = 0.
$$
For the term $B(z)$ we perform a capacity estimate. First, note that
\begin{equation*}
B(z)\le 2 \int_{|w|<2|z|} \log \frac{1}{|w-z|} \frac{d\nu(w)}{|w|},
\end{equation*}
and
\begin{equation*}
\begin{split}
 \int_{|w|<2|z|} \frac{d\nu(w)}{|w|}&=\left[\frac{\nu B(0,\rho)}{\rho}\right]_0^{2|z|} +
\int_{0}^{2|z|} \frac{\nu B(0,\rho)}{\rho^{2}}\,d\rho\\*[7pt]
&\le \frac{\|\nu\|_{\infty}}{\log^2(2|z|)} +\int_0^{2|z|} \frac{\eta(\rho)}{\log^2(\rho)} \frac{d\rho}{\rho}\\*[7pt]
&\le \frac{\|\eta\|_{\infty}}{\log^2(2|z|)}+\frac{1}{\log\frac{1}{2|z|}}\sup_{\rho<2|z|} \eta(\rho).
\end{split}
\end{equation*}
Therefore
\begin{equation*}
\begin{split}
\sup_{t>0} \frac{t\operatorname{Cap}\left\{ z\in B(0,r): B(z)>t\right\}}
{\operatorname{Cap}(B(0,r))}
&\le \frac{1}{\operatorname{Cap}(B(0,r))}
\left(\frac{\|\eta\|_{\infty}}{\log^2(2r)}+\frac{1}{\log \frac{1}{2r}}
\sup_{t<2r}\eta(t)\right)\\*[7pt]
&\le C\left\{ \frac{\|\eta\|_{\infty}}{\log\frac{1}{2r}} +\sup_{t<2r}\eta(t)\right\},
\end{split}
\end{equation*}
which tends to $0$ with $r.$  It is worth remarking that only in the
last inequality we used that $\eta(r)$ tends to $0$ with $r$.\qed

\smallskip
Theorem \ref{teo4} is sharp in the scale of Hausdorff
measures, as the next result shows.
\begin{theorem}\label{teo7}
 Let $\Psi \colon[0,\infty) \rightarrow [0,\infty)$, $\Psi(0)=0$, be a continuous (strictly)  increasing function such that
 \begin{equation}\label{limsuphpsi}
 \limsup_{t \rightarrow 0} \frac{\Psi(2r)}{\Psi(r)}< 4,
 \end{equation}
 and
 \begin{equation*}\label{Mpsi}
  M(r): = \frac{\Psi(r)}{\psi(r)} \rightarrow \infty, \quad \text{as} \quad r \rightarrow \infty.
 \end{equation*}
Then there exists a compact set $K$ with $H^\Psi(K) >0$ and a finite
Borel measure whose logarithmic potential is not differentiable in
the capacity sense at $H^\Psi$ almost all points of~$K$.
\end{theorem}

Therefore there is no condition better than $H^\psi(E)=0$ on the set
$E$ of points of non differentiability in the capacity sense for the
logarithmic potential of a finite Borel measure. In particular,
there exists a finite Borel measure whose logarithmic potential is
not differentiable in the capacity sense on a set of positive
$H^\varphi$ measure. Thus the two notions of differentiability in the capacity sense are different in dimension $2$.
Also note that the size of the exceptional sets is
definitely larger in dimension $2$
 than in higher dimensions. See subsection 2.5 for a discussion of condition \eqref{limsuphpsi}.

\begin{proof}[Proof of Theorem \ref{teo7}]
The proof is similar to that of Theorem \ref{teo6}, although a
difficulty appears that requires a new idea. The proof is written to
make it accessible to a reader who has not gone through the proof of
Theorem \ref{teo6}.

 Let $K$ be the Cantor set associated with
$\Psi$ and let $\mu$ be its canonical measure. We aim at
constructing a finite Borel measure $\nu$ whose logarithmic
potential is not differentiable in the capacity sense at $\mu$
almost all points.

Given a positive integer $n$ take another large positive integer
$N_n$ to be determined later. Given a square $Q^{N_n}_j$ of
generation $N_n$ let $Q^{2N_n}_j, \; 1 \le j \le 4^{N_n}, $ the
squares of generation $2N_n$ contained in $Q^{N_n}_j$. Choose a
square of generation $2 N_n +n$ inside $Q^{2N_n}_j$ and let $p^n_j$
be its center. It is not important what square is chosen; what matters
is that it is a square of generation $2 N_n+n$. Descending to generation $2N_n$ instead of $N_n$
is a first difference with respect to the proof of Theorem \ref{teo6}. It will become apparent later why we need to do so.  Set
$$
E_n = \bigcup_{j=1}^{4^{2 N_n}} B(p^n_j, \sqrt{2}\,\sigma_{2
N_n+n}),
$$
$$
D_m =\bigcup_{n=m}^{\infty} E_n,
$$
and
$$
D= \bigcap_{m=1}^\infty D_m.
$$
Clearly $\mu(E_n)= 4^{2 N_n} 4^{-(2 N_n+n)}=4^{-n}$, and this is the
only reason why we have descended $n$ more generations after $2
N_n$. Hence $\mu(D_m) \le \sum_{n=m}^\infty 4^{-n}$ and $\mu(D)=0.$

Take $a \in K \setminus D.$ Then $a \neq p^n_j, \; $ for all $n$ and
$j$, because $p^n_j \notin K.$ Since $a \notin D$, $a \notin D_m$
for some $m$, and so $a \notin B(p^n_j, \sqrt{2}\,\sigma_{2 N_n+n})$
for all $n \ge m$ and all $j.$

We proceed now to define the finite Borel measure whose logarithmic
potential is not differentiable in the capacity sense at all points
$a \in K \setminus D.$ Set
\begin{equation*}
\nu = \sum_{n=1}^\infty \frac{1}{n^2 \,4^{2 N_n}} \sum_{j=1}^{4^{2
N_n}} \left(\delta_{p^n_j} - \frac{d \sigma_{p^n_j,\sqrt{2}\,\,
\sigma_{2 N_n+n}}}{2 \pi \sqrt{2}\, \sigma_{2 N_n+n}} \right),
\end{equation*}
where $\delta_p$ is the Dirac delta at the point $p$ and
$d\sigma_{p,\rho}$ is the arc length measure on $\partial
B(p,\rho).$ Since $\|\nu\| \le 2 \sum_{n=1}^\infty 1/ n^2, \,$ $\nu$
is a finite Borel measure. The logarithmic potential of $\nu$ is
\begin{equation*}
u(z) = \sum_{n=1}^\infty \frac{1}{n^2 \,4^{2 N_n}} \sum_{j=1}^{4^{2
N_n}} L(\frac{1}{\sqrt{2}\, \sigma_{2 N_n+n}} (z-p^n_j)),
\end{equation*}
where $L$ is the function in \eqref{L}.
To simplify notation write
\begin{equation*}
S_n(z) = \frac{1}{4^{2 N_n}} \sum_{j=1}^{4^{2 N_n}}
L\left(\frac{1}{\sqrt{2}\, \sigma_{2 N_n+n}} (z-p^n_j)\right).
\end{equation*}
Given $a \in K \setminus D$ as before,  we have $a \notin B(p^n_j,
\sqrt{2}\,\sigma_{2 N_n+n})$ for all $n \ge m$ and all $j.$ Thus
$S_n(a)=0$, for all $n \ge m$ and consequently
\begin{equation*}
u(z)-u(a) = \sum_{n=1}^{m-1} \frac{1}{n^2}
\left(S_n(z)-S_n(a)\right) + \sum_{n=m}^{\infty} \frac{1}{n^2}
S_n(z).
\end{equation*}
Recall that  $a \neq p^n_j, \; $ for all $n$ and $j$. If $r>0 $ is
small enough, then $p^n_j \notin B(a,r)$, for $ n \le m-1$ and all
$j.$ Therefore
$$
 \sum_{n=1}^{m-1} \frac{1}{n^2}
\left(S_n(z)-S_n(a)\right)
$$
is smooth on $B(a,r).$  Consequently the differentiability
properties of $u$ at the point $a$ depend only on
$$
R(z) : = \sum_{n=m}^{\infty} \frac{1}{n^2} S_n(z).
$$

 Assume that $R$ is differentiable in the capacity
sense at $a$. Then $R$ is Lipschitz in the
capacity sense at the point $a$, as a simple argument shows. Since $R(a)=0$ this means that
\begin{equation*}
\frac{\sup_{t > 0} t \operatorname{Cap}(\{z \in B(a,r) : \frac{|R(z)|}{|z-a|}
> t \})}{\operatorname{Cap}(B(a,r))} \le C_a, \quad 0 < r < 1/4,
\end{equation*}
for some constant $C_a$ depending only on $a.$  To disprove the
preceding inequality we take radii of the form
\begin{equation*}
r=r_k= \sqrt{2}\,\sigma_{N_k},\quad k =1, 2, \dots
\end{equation*}
with $k \ge m$ large enough so that $p^n_j \notin B(a,\sqrt{2}\,\sigma_{N_k})$, for $ n
\le m-1$ and all $j.$
For each such $k$ the point $a$ belongs to a square $Q^{N_k}$ of
generation $N_k$, which contains $4^{N_k}$ points~$p^k_j$.
Now we classify the $p_j^k \in Q^{N_k}$ according to their distance to $a.$
Denote by  $Q^{N_k+1}$ a square of generation $N_k+1$ contained in $
Q^{N_k}$ and not containing $a.$  The square $Q^{N_k+1}$ contains
$4^{N_k-1}$ points $p_j^k.$ If $p_j^k \in Q^{N_k+1}$ and $z \in
B(p_j^k, \sqrt{2}\, \sigma_{2 N_k+k})$, then $|z-a| < 2
\sigma_{N_k}.$ We construct inductively pairwise disjoint squares
$Q^{N_k+l}, \, l=1,2,...,N_k$, of generation $N_k+l$, contained in
$Q^{N_k}$, containing $4^{N_k-l}$ points $p_j^k$, and with the
property that if $p_j^k \in Q^{N_k+l}$ and $z \in B(p_j^k,
\sqrt{2}\, \sigma_{2 N_k+k})$, then $|z-a| < 2 \sigma_{N_k+l-1}.$
Since
\begin{equation*}
 B(p^k_j, \sqrt{2}\,\sigma_{2 N_k+k}) \subset B(a,
\sqrt{2}\,\sigma_{N_k} )=B(a,r), \quad p^k_j \in Q^{N_k},
\end{equation*}
and
\begin{equation*}
R(z) \ge \frac{1}{k^2} S_k(z) \ge \frac{1}{k^2 4^{2
N_k}}\,L(\frac{1}{\sqrt{2}\, \sigma_{2 N_k+k}} (z-p^k_j)),
\end{equation*}
we get
\begin{equation*}
\begin{split}
&\left\{z  \in B(a,r) :  \frac{R(z)}{|z- a|}> t \right\}\\
&\qquad\supset
\bigcup_{l=1}^{N_k} \,\bigcup_{p_j^k \in Q^{N_k+l}} \left\{ z \in
B(p^k_j, \sqrt{2}\, \sigma_{2 N_k+k}) : \log \frac{\sqrt{2}\,
\sigma_{2 N_k+k}}{|z-p_j^k|} > t\, k^2 \,4^{2N_k} \,2\,
\sigma_{N_k+l-1} \right\} \\
&\qquad
= \bigcup_{l=1}^{N_k} \,\bigcup_{p_j^k \in Q^{N_k+l}} B_{lj},
\end{split}
\end{equation*}
where
\begin{equation*}\label{unio2}
B_{lj} =  B \left(p_j^k, \sqrt{2}\, \sigma_{2 N_k+k} \; e^{-t\, 2\,
k^2\, 4^{2N_k} \,\sigma_{N_k+l-1}} \right), \quad p^k_j \in Q^{N_k+l}.
\end{equation*}
 Lemma \ref{lemma8} below yields that if $t > T_k$ for a large positive
number $T_k$, then the balls $B_{lj}$ are disjoint and
\begin{equation}\label{capunio}
\Capa \left(\bigcup_{l=1}^{N_k} \,\bigcup_{p_j^k \in Q^{N_k+l}}
B_{lj} \right) \ge \frac{1}{2}\,\sum_{l=1}^{N_k} \sum_{p_j^k \in
Q^{N_k+l}} \Capa \left(B_{lj}\right).
\end{equation}
The proof of Lemma \ref{lemma8} will be discussed later.  It seems worthwhile to make a digression now to explain
the need to descend to generation $2N_k$. Should we have proceeded as in the proof of Theorem \ref{teo6}
we would have descended up to generation $N_k$ only, which means taking only one term in the union in
the left hand side of \eqref{capunio}. Thus we would have obtained
\begin{equation*}
\begin{split}
\frac{\sup_{t > 0} t  \Capa (\{ z  \in B(a,r) :  \frac{R(z)}{|z- a|} > t \}) } {\Capa(B(a,r))}
&\ge\log \left(\frac{1}{r}  \right)\sup_{t > 0}  \frac{t}{t 2k^2 4^{N_k}\sigma_{N_k}+\log 1/ \sqrt{2} \,\sigma_{N_k+k}}\\*[7pt]
&\ge c \log \left(\frac{1}{\sigma_{N_k}}\right) \frac{1}{k^2 4^{N_k} \sigma_{N_k}} \ge \frac{c}{k^2} \frac{M(\sigma_{N_k})}{\log 1/\sigma_{N_k}},
\end{split}
\end{equation*}
which does not conclude.

We proceed to complete the proof using Lemma \ref{lemma8}. We have
\begin{equation*}\label{capgt}
\begin{split}
\sup_{t > 0} t  \Capa (\{ z & \in B(a,r) :  \frac{R(z)}{|z- a|} > t \})  \\
&  \ge c \,\sup_{t > T_k}  \sum_{l=1}^{N_k} \sum_{p_j^k \in
Q^{N_k+l}} \frac{t}{t\, 2\, k^2\, 4^{2N_k} \,\sigma_{N_k+l-1} +
\log\frac{1}{\sqrt{2}\, \sigma_{2 N_k+k}}}   \\ &  \ge c \,\sup_{t >
T_k} \sum_{l=1}^{N_k} \frac{t \, 4^{N_k-l}}{t\, 2\, k^2\, 4^{2N_k}
\,\sigma_{N_k+l-1} + \log\frac{1}{\sqrt{2}\, \sigma_{2 N_k+k}}}
\\& = \frac{c}{k^2} \, \sum_{l=1}^{N_k} \frac{ 4^{N_k-l}}{ 4^{2N_k}
\,\sigma_{N_k+l-1}} \\& = \frac{c}{k^2} \, \sum_{l=1}^{N_k}
\frac{1}{ 4^{N_k+l-1} \,\sigma_{N_k+l-1}} \\ & = \frac{c}{k^2} \,
\sum_{l=1}^{N_k} \frac{M(\sigma_{N_k+l-1})}{ \log^2
\left(\frac{1}{\sigma_{N_k+l-1}}\right)} \\ & \ge \frac{c}{k^2} \,
\inf_{N \ge N_k} M(\sigma_N) \sum_{l=1}^{N_k} \frac{1}{ \log^2
\left(\frac{1}{\sigma_{N_k+l-1}}\right)}
\end{split}
\end{equation*}
and so, recalling that $r=r_k= \sqrt{2} \sigma_{N_k}$,
\begin{equation}\label{capgt1}
\begin{split}
&\frac{\sup_{t > 0} t  \operatorname{Cap} (\{ z 
\in B(a,r) :
\frac{R(z)}{|z-a|}>t \})}
{\operatorname{Cap}(B(a,r))} \\
&\hspace*{2.50cm} \ge \frac{c}{k^2} \, \inf_{N \ge N_k}M(\sigma_N)\, \log
\left(\frac{1}{\sigma_{N_k}}\right)
\,\sum_{l=1}^{N_k}
\frac{1}{\log^2\left(\frac{1}{\sigma_{N_k+l-1}}\right)}.
\end{split}
\end{equation}
At this point it is convenient to distinguish two cases. The first
is that
\begin{equation}\label{cas1}
\lim_{n \rightarrow \infty}\frac{\Psi(\sigma_n)}{\sigma_n} =0.
\end{equation}
Let us check that then, for some positive integer $n_0$,
\begin{equation}\label{log}
n \,\log 2  \le \log \frac{1}{\sigma_n} \le  n \,\log 4, \quad n \ge
n_0.
\end{equation}
The first inequality follows from the definition of Cantor sets which
gives $\sigma_n < 2^{-n}$ for all~$n.$ The second follows from
\eqref{cas1}, which yields $ 4^n \sigma_n \ge 1, \; n\ge n_0.$
Introducing \eqref{log} in~\eqref{capgt1} one gets
\begin{equation*}
\frac{\sup_{t > 0} t  \Capa (\{ z  \in B(a,r) :  \frac{R(z)}{|z- a|}
> t \})}{\Capa(B(a,r))} \ge \frac{c}{k^2} \, \inf_{N \ge N_k}
M(\sigma_N),
\end{equation*}
and now it only remains to choose $N_k$ large enough so that
\begin{equation*}\label{capinfinit}
 \inf_{N \ge N_k} M(\sigma_N) \ge k^3.
\end{equation*}

If \eqref{cas1} is not satisfied then for some $\delta >0$ and for infinitely many indexes
$n$ one has $\Psi(\sigma_n)/\sigma_n \ge \delta > 0.$ Given $x$ in
the Cantor set $K$ let $Q^n$ the square of generation $n$ containing
$x.$ Then for the measure $\mu$ associated with $K$ we have
\begin{equation*}\label{densitatpos}
\frac{\mu(B(x, \sqrt{2}\,\sigma_n))}{\sqrt{2}\,\sigma_n} \ge
\frac{\mu(Q^n)}{\sqrt{2}\,\sigma_n} = \frac{1}{\sqrt{2}}
\frac{\Psi(\sigma_n)}{\sigma_n} \ge \frac{\delta}{\sqrt{2}},
\end{equation*}
which says that $\mu$ has no vanishing linear density at any point
of $K.$ Thus the logarithmic potential of $\mu$ is not
differentiable in the capacity sense at any point of $K$ and we are
done in this case without resorting to any complicated measure like
$\nu$.
\end{proof}

We turn now to the discussion of inequality \eqref{capunio}.

\begin{lemma}\label{lemma8}
Let $B_j = B(p_j, r_j), \; 1 \le j\le N,$ a family of disjoint discs
of center $p_j$ and radius $r_j < 1.$  Let $\delta = \min\limits_{j \ne k}
\operatorname{dist}(B_j, B_k)$ and assume that $0 < \delta < 1.$ Set
$\sigma = \max\limits_{j} r_j.$  If $\sigma \le \delta^N$, then
\begin{equation}\label{capunio2}
\Capa \left(\bigcup_{j=1}^{N} B_j\right) \ge
\frac{1}{2}\,\sum_{j=1}^{N} \Capa (B_j).
\end{equation}
\end{lemma}

To apply Lemma \ref{lemma8} to \eqref{capunio} note that the radius
of the disc $B_{lj}$ is
$$
\sqrt{2}\, \sigma_{2 N_k+k} \; e^{-t\, 2\, k^2\, 4^{2N_k}
\,\sigma_{N_k+l-1}} \le  e^{-t\, 4^{2N_k} \,\sigma_{2 N_k}}
$$
and the distance between two such discs is larger than
$\sigma_{2N_k-1}-2 \sigma_{2 \,N_k}>0.$ For any fix $k$ the number
of discs $B_{lj}$ is less than $4^{N_k}.$ Hence the hypothesis of
Lemma~\ref{lemma8} are satisfied if
$$
t \ge T_k := \frac{4^{N_k}}{4^{2N_k}\,\sigma_{2N_k}} \log
\frac{1}{\sigma_{2N_k-1}-2 \sigma_{2 \,N_k}},
$$
which is the large number $T_k$ used in the proof of Theorem \ref{teo7}.
\begin{proof}[Proof of Lemma \ref{lemma8}]
The normalized equilibrium potential of the disc $B_j=B(p_j,r_j)$ is
\begin{equation*}\label{eqpot}
 u_j = \frac{1}{\log \frac{1}{r_j}}\,\log \frac{1}{|z|} * \frac{d\sigma_j}{2\pi r_j},
\end{equation*}
where $\sigma_j$ stands for the arc-length measure on $\partial B_j$. Then
\begin{equation*}\label{eqpot2}
u_j(z)=
\begin{cases}
\frac{1}{\log \frac{1}{r_j}}\,\log \frac{1}{|z-p_j|} &\text{if}\quad |z-p_j| \ge r_j,\\*[11pt]
 \hspace{1cm}1  & \text{if}\quad |z-p_j| \le r_j.
\end{cases}
\end{equation*}
If $z \in B_k, \, k \neq j$ then
\begin{equation*}\label{ujota}
 u_j(z) \le \frac{\log \frac{1}{\delta}}{\log \frac{1}{\sigma}},
\end{equation*}
and so
\begin{equation*}\label{ujota}
 \sum_{j=1}^N u_j(z) \le 1+ (N-1)\frac{\log \frac{1}{\delta}}{\log \frac{1}{\sigma}} \le 1+\frac{N-1}{N} \le 2,
 \quad z \in \C,
\end{equation*}
which yields \eqref{capunio2} by definition of Wiener capacity \eqref{newtoniancap}.
\end{proof}

\section{Second order differentiability}
\begin{proof}[Proof of Theorem \ref{teo5}, part (i)]  Assume that $d \ge 3$. Then the first order derivatives of $1/|x|^{d-2}$
in the distributions sense are
the locally integrable functions
\begin{equation*}\label{prider}
 \partial_i\, \frac{1}{|x|^{d-2}} = -(d-2) \frac{x_i}{|x|^{d-2}}, \quad 1 \le i \le d.
\end{equation*}
The second order derivatives in the distributions sense are given by principal value distributions and the Dirac delta $\delta_0$
at the origin via the identities
\begin{align*}
 \partial_{ij} \, \frac{1}{|x|^{d-2}} &= d(d-2) \operatorname{p.v.} \frac{x_i x_j}{|x|^{d+2}},
 \quad i \neq j, \\*[7pt]
  \partial_{ii} \, \frac{1}{|x|^{d-2}} &= -(d-2) \operatorname{p.v.} \frac{|x|^2- d \,x_i^2}{|x|^{d+2}}
+a_d \,\delta_0,
\end{align*}
where $a_d = -(d-2) \omega_{d-1} /d$ and $\omega_{d-1}$ is the
$(d-1)$-dimensional surface measure of the unit sphere in $\Rd.$

Assume that $\varphi$ is a $C^{\infty}$ function with compact support. Then $u = 1/|x|^{d-2}*\varphi$ is a
$C^{\infty}$ function on $\Rd$ and its second order partial derivatives are
\begin{alignat*}{3}\label{segderfi}
& \partial_{ij} \,u(x) = d(d-2) \left(\operatorname{p.v.} \frac{x_i x_j}{|x|^{d+2}} * \varphi \right) (x), &\quad &i \neq j, &\quad &x \in \Rd, \\*[7pt]
&  \partial_{ii} \,u(x) =  -(d-2) \left(\operatorname{p.v.} \frac{|x|^2- d \,x_i^2}{|x|^{d+2}} *  \varphi \right)
(x)+a_d \,\varphi (x), &\quad &1 \le i \le d, &\quad &x \in \Rd.
\end{alignat*}
In particular, the principal value integrals exist at each point $x \in \Rd$.

Given a finite Borel measure $\mu$ in $\Rd$, there is a way of defining first and second derivatives
of the potential $u= 1/|x|^{d-2}*\mu$ at a fixed point $a \in \Rd$. For the first order derivatives we only have to
require that $a$ is a Lebesgue point of the locally integrable functions
\begin{equation*}\label{firstdermu}
\partial_i u = -(d-2) \frac{x_i}{|x|^{d-2}}* \mu , \quad 1 \le i \le d.
\end{equation*}
For the second order derivatives
\begin{equation}\label{segdermu}
\begin{aligned}
 \partial_{ij} \,u &= d(d-2) \left(\operatorname{p.v.} \frac{x_i x_j}{|x|^{d+2}} * \mu \right),
& \quad &i \neq j, \quad x \in \Rd, \\*[7pt]
  \partial_{ii} \,u &= -(d-2) \left(\operatorname{p.v.} \frac{|x|^2- d \,x_i^2}{|x|^{d+2}} *  \mu \right)
+a_d \,\mu , &\quad &1 \le i \le d,
\end{aligned}
\end{equation}
it is natural to require existence at the point $a$ of all the above principal value integrals and of
the limit
\begin{equation}\label{mutilda}
\tilde{\mu} (a): = \lim_{r \rightarrow 0} \frac{\mu(B(a,r))}{r^{d}}.
\end{equation}
 We know, by Lebesgue differentiation theorem and by standard Calder\'{o}n--Zygmund theory,
that the stated conditions are satisfied for almost all points $a$ with respect to $d$ dimensional
Lebesgue measure $dx$. Then the prospective second order Taylor polynomial of $u$ at $a$
\begin{equation}\label{secondtaylor}
u(a)+ \sum_{i=1}^d \partial_i u(a) (x_i-a_i)+ \frac{1}{2} \, \sum_{i,j=1}^d \partial_{ij}u(a)(x_i-a_i)(x_j-a_j)
\end{equation}
is defined at almost all points.

Now we make a convenient reduction. To study differentiability properties of $u$
at a fixed point $a$ it is enough to replace $\mu$ by $\chi_B \mu$, with $B=B(a,1)$, because the potentials
of $\mu$ and $\chi_B \mu$ differ by a smooth function on $B$. Let $\varphi \in C^{\infty}$ be a function
with compact support in the ball $B(a,2)$ taking the value $1$ on $B$. Then by the Radon--Nikodym decomposition
there is a function $f$ in $L^1(B)$ such that
\begin{equation*}\label{lebu}
\mu = (f-f(a)) \varphi\, dx + \mu_s + f(a) \varphi \,dx,
\end{equation*}
where $\mu_s$ is the singular part of $\mu$. Since the potential of $\varphi dx$ is smooth on $\Rd$,
we can assume that $\mu$ is a positive measure which satisfies
\begin{equation*}\label{zerodensityleb}
\tilde{\mu} (a): = \lim_{r \rightarrow 0} \frac{\mu(B(a,r))}{r^{d}}=0.
\end{equation*}
One of the effects of this assumption is that in the definition of the second order
derivatives~$\partial_{ii} \,u $ at the point $a$ one can avoid the second term in \eqref{segdermu},
which would be the limit~\eqref{mutilda}.

We have to show \eqref{defdifcap2} where $D(x)$ is as in \eqref{Q2} with the second order
Taylor polynomial  as in \eqref{secondtaylor}. The structure of the proof is very similar to that of the
sufficiency part in Theorem \ref{teo1}, so we only outline the argument. Take $a=0$ for simplicity.
First we replace the principal
value integrals by truncations at level $\varepsilon$, where $\varepsilon = 2|x|$. The difference is a term which
tends to $0$ with $\varepsilon.$  We split the domain of integration of the integral into two pieces,
one corresponding to $|y|> \varepsilon$. In that piece one estimates the remainder of the Taylor expansion up to
order $2$ in terms of third derivatives. The upper bound one gets is
\begin{equation*}\label{mgep}
C\,\varepsilon\, \int_{|y|>\varepsilon} \frac{ d\mu(y)}{|y|^{d+1}}.
\end{equation*}
This term is estimated by integration by parts introducing a parameter $N$ as is \eqref{Bmp}. It remains
to estimate the integral over $|y| < \varepsilon$ with respect to $\mu$ of
\begin{equation*}\label{fop}
 \frac{1}{|x|^2} \left| \frac{1}{|x-y|^{d-2}} -\frac{1}{|y|^{d-2}} -(d-2)\left\langle
\frac{y}{|y|^d},x\right\rangle\right|,
\end{equation*}
which is not greater than a constant times the sum of the $3$ terms
\begin{equation*}\label{trestermes}
 \frac{1}{\varepsilon^2} \frac{1}{|x-y|^{d-2}} +\frac{1}{\varepsilon^2}  \frac{1}{|y|^{d-2}}+  \frac{1}{\varepsilon}
\frac{1}{|y|^{d-1}}.
\end{equation*}
The integral over $|y| < \varepsilon$ with respect to $d\mu$ of the second and third terms above is less
than or equal to
\begin{equation*}\label{segonitercer}
 \frac{1}{\varepsilon}  \int_{|y|<\varepsilon} \frac{1}{|y|^{d-1}}\,d\mu(y),
\end{equation*}
which is estimated by an integration by parts as in \eqref{Bmp}. The upper bound one gets is
\begin{equation*}\label{segonitercerbound}
 C \, \sup_{0 < \rho< \varepsilon} \frac{\mu(B(0,\rho))}{\rho^d} \rightarrow 0,
 \quad \text{as}\quad \varepsilon \rightarrow 0.
\end{equation*}
One is left with
\begin{equation*}\label{primer}
F(x) = \frac{1}{\varepsilon^2} \int_{|y|<\varepsilon}  \frac{1}{|x-y|^{d-2}} \, d\mu(y).
\end{equation*}
If $x \in B(0,r)$ then
\begin{equation*}\label{efamaj}
F(x) \le P\left(\chi_{B(0,2r)}(y) \frac{d\mu(y)}{|y|^2}\right) (x),
\end{equation*}
where $P$ is the Newtonian potential of the indicated measure, as in \eqref{pot}. The total mass of the measure
$\chi_{B(0,2r)}(y) \,d\mu(y)/ |y|^2$ is estimated
by an integration by parts and one gets the upper bound
\begin{equation*}\label{massamusobre2}
 \sup_{0 < \rho< r} \frac{\mu(B(0,\rho))}{\rho^d} \, r^{d-2}.
\end{equation*}
Therefore
\begin{equation*}\label{capF}
\sup_{t>0} \frac{t\operatorname{Cap}\left\{x\in B(0,r): F(x)>t\right\}}{\operatorname{Cap}(B(0,r))} \le C\,
\sup_{0 < \rho< r} \frac{\mu(B(0,\rho))}{\rho^d},
\end{equation*}
which completes the proof of Theorem \ref{teo5}, part (i).
\end{proof}

\begin{proof}[Proof of Theorem \ref{teo5}, part (ii)] The construction is practically that of
Calder\'{o}n in \cite{C}, so that we will briefly outline the argument. It is enough to construct a finite Borel
measure in the unit square $Q =[0,1] \times[0,1]$ whose logarithmic potential is not second order
differentiable in the weak capacity sense at almost all points of the square. If this measure has been
constructed, then one covers the plane
by disjoint dyadic squares~$Q_n$, $n=0,1, \dots$ of side  length $1$ and one sets $\mu = \sum_{n=0}^\infty 1/2^n \,\mu_n$, where $\mu_n$ is
the translation into $Q_n$  of the measure constructed in the unit square.

\enlargethispage{3mm}

Divide the unit square into $4^{n^2}$ disjoint squares of side length $2^{-n^2}$. The vertices of those squares
not lying in the boundary of the unit square are of the form $(i 2^{-n^2}, j 2^{-n^2})$ with $1 \le i, j \le
(2^{n^2}-1)$. There are $N_n : = (2^{n^2}-1)^2 \le 4^n$ such vertices. Denote them by $p_{nk}$, where the
index $k$ varies from $1$ to $N_n.$ Let $B_{nk}$ be the ball with center $p_{nk}$
and radius  $1/(n 2^{n^2})$. Set
\begin{equation*}\label{esubn}
 E_n = \bigcup_{k=1}^{N_n} B_{nk}, \quad\quad D_m = \bigcup_{n=m}^{\infty} E_n, \quad\quad
 D= \bigcap_{m=1}^\infty D_m,
\end{equation*}
so that $|E_n| \le N_n / (n^2 4^{n^2}) \le 1/n^2$. Hence $|D_m| \rightarrow 0$ as $m \rightarrow \infty$
and $|D|=0.$ Let $L$ be the function in \eqref{L}. Define
\begin{alignat*}{2}
 S_n(z) &= \frac{1}{N_n} \sum_{k=1}^{N_n}
L(n 2^{n^2} (z-p_{nk})), &\quad &z \in \C,\\*[7pt]
 u(z) &= \sum_{n=1}^\infty \frac{1}{n^{3/2}} \,S_n(z), &\quad &z \in \C.
\end{alignat*}
Then $u$ is the logarithmic potential of a finite Borel measure supported in the unit square.

We claim that $u$ is not second order differentiable in the weak capacity sense at any point of
$Q \setminus D$. Take $a \in Q\setminus D$,  so that $a \in Q \setminus D_m$ for some positive integer $m.$
Thus $a \in Q,$ $a \notin B_{nk}, \; n \ge m, \; 1 \le k \le N_n.$  We consider radii of the form $r=r_q =
1/2^{q^2}$. If $q$ is large enough then $r$ is small enough so that the ball $B(a,r)$ does not contain any~$p_{nk}$, $1 \le n \le m-1, \; 1 \le k \le N_n$. Then
\begin{equation*}\label{uminusgrad}
\begin{split}
 u(z)-&u(a)-  \langle \nabla u(a), z-a \rangle \\*[7pt]
 & = \sum_{n=1}^{m-1} \frac{1}{n^{3/2}} \left(S_n(z)-S_n(a)
 -\langle \nabla S_n(a), z-a \rangle \right)+ \sum_{n=m}^{\infty} \frac{1}{n^{3/2}} S_n(z)
\end{split}
\end{equation*}
and the first term in the right-hand side  is smooth on $B(a,r)$. Hence the differentiability
properties of $u$ are exactly those of
$$
R(z) : = \sum_{n=m}^{\infty} \frac{1}{n^{3/2}} S_n(z).
$$
 Assume that $R$ is second order differentiable in the weak capacity
sense at $a$. Then
\begin{equation}\label{restalipcap}
\frac{\sup_{t > 0} t \operatorname{Cap}(\{z \in B(a,r) : |R(z)|
> t \})}{r^2 \,\operatorname{Cap}(B(a,r))} \le C_a, \quad 0 < r < 1/4,
\end{equation}
for some constant $C_a$ depending only on $a.$ To disprove \eqref{restalipcap} we note that, since $a \in Q$,
there is  a point  $p_{qk} \in B(a,1/2^{q^2}).$  Moreover $|p_{qk} -a| < (1/\sqrt{2})r, \; r=1/2^{q^2}.$

If $q \ge m$, then
\begin{equation*}\label{erramg}
R(z) \ge \frac{1}{q^{3/2}} S_q(z) \ge \frac{1}{q^{3/2}
N_q}\,L(q 2^{q^2} (z-p_{qk})).
\end{equation*}
If $q$ is large enough the set $\{z \in B(a,r) : |R(z)|
> t \}$ contains the ball of center $p_{qk}$ and radius $1/(e^{t q^{3/2} 4^{q^2}} q 2^{q^2})$. Thus the
left-hand side of \eqref{restalipcap} is not less than a constant times
\begin{equation*}\label{caperramg}
q^2 4^{q^2} \sup\limits_{t>0}
\frac{t}{q^{3/2} 4^{q^2} t + \log q \,+ q^2 \log 2 }\\
 =  q^2 4^{q^2} \frac{1}{q^{3/2} 4^{q^2}} = q^{1/2},
\end{equation*}
which shows that \eqref{restalipcap} cannot hold.
\end{proof}

\section{The equilibrium measure}

For each compact $E$ subset of $\mathbb{R}^{d}$, $d\ge 2$, there exists a unique probability measure~$\mu$ supported on~$E$ of minimal energy. In other words, the infimum in~\eqref{capenergy} is attained by~$\mu$. This probability measure is called the equilibrium measure and it can be shown that its potential (the equilibrium potential) is constant on~$E$ except for a set of zero Newtonian capacity (Wiener capacity for $d=2$). In this section we present a proof of the following result, due to Oksendal in the plane; see ~\cite[Corollary 1.5]{O1} and \cite{O2}. An  alternative proof which works in higher dimensions for harmonic measure is in \cite[Theorem 10] {GS}.

\begin{theorem}\label{teo8}
The equilibrium measure of a compact subset of~$\mathbb{R}^{d}$ is singular with respect to $d$-dimensional Lebesgue measure.
\end{theorem}

\begin{proof}[Proof in $\mathbb{R}^{d}$, $d\ge 3$]
We plan to apply Theorem~\ref{teo5}.

Set $u=\frac{1}{|x|^{d-2}} *\mu$. By Theorem~\ref{teo5} we have \eqref{defdifcap2} at almost all points $a\in\mathbb{R}^{d}$. Set $\nabla u(a)=(A_{1},\dotsc,A_{d})$ and let $B$ stand for the symmetric $d\times d$~matrix with entries~$B_{ij}$. Here the $A_{i}$ and the $B_{ij}$ are as in \eqref{Q2}. Set $\mu =f\,dx+\mu_{s}$, with $f\in L^{1}(dx)$ and $\mu_{s}$ singular with respect to~$dx$. Thus, by \eqref{segdermu},
\begin{equation}\label{eq54}
\sum_{i=1}^{d}B_{ii}=d\,a_{d}f(a).
\end{equation}

\begin{lemma}
The set of points~$a \in E$ where $u$ is second order differentiable in the capacity sense and $\nabla u(a)\ne 0$ is a countable union of sets of finite $(d-1)$-dimensional Hausdorff measure.
\end{lemma}

\begin{proof}
Since the equilibrium potential~$u$ is constant $\operatorname{Cap}$-a.e.\ on~$E$ we have, by Theorem~\ref{teo5},
\begin{equation}\label{eq55}
\lim_{r\to 0}\frac{1}{\operatorname{Cap}B(a,r)} \sup_{t>0}t\operatorname{Cap} \left\{x\in B(a,r)\cap E:\frac{|\langle \nabla u(a),x-a\rangle +\langle B(x-a),x-a\rangle|}{|x-a|^{2}} > t \right\} =0 ,
\end{equation}
for almost all points $a \in E$.
Assume that $a=0$, $\nabla u(0)\ne 0$ and, without loss of generality, that $\nabla u(0)=\lambda (0,\dotsc,0,1)$, with $\lambda >0$. Given $\delta>0$ consider the cone
\begin{equation}\label{eq56}
K_{\delta}=\left\{ x\in\mathbb{R}^{d}\backslash \{0\}:\lambda \frac{|x_{d}|}{|x|}>\delta\right\}.
\end{equation}
If $x\in B(0,r)\cap E\cap K_{\delta}$ and $r$ is small enough we have, for some positive constant $C$,
$$
\frac{|\langle \nabla u(0),x\rangle +\langle Bx,x\rangle|}{|x|^{2}} \ge \lambda \frac{|x_{d}|}{|x|^{2}}-C\ge \frac{\delta}{|x|}-C \ge \frac{\delta / 2}{|x|}.
$$
Taking $t=1$ and $r<\delta / 2$ we get
$$
\lim_{r\to 0}\frac{\operatorname{Cap}(B(0,r)\cap E\cap K_{\delta})}{r^{d-2}}=0.
$$
Since one has the general inequality $\operatorname{Cap}(F)^{\frac{1}{d-2}}\ge c_{d}\, H_{\infty}^{d-1} (F)^{\frac{1}{d-1}}$ relating capacity and $(d-1)$-dimensional Hausdorff content of compact sets~$F$ (\cite[Corollary 5.1.14]{AH}), we conclude that
$$
\lim_{r\to 0}\frac{H_{\infty}^{d-1} (B(0,r)\cap E\cap K_{\delta})}{r^{d-1}}=0,
$$
which means that the hyperplane~$x_{d}=0$ is an approximate tangent hyperplane to~$E$ at~$0$. The set of points of~$E$ where there exists such a tangent hyperplane is a countable union of sets with finite $(d-1)$-dimensional Hausdorff measure (\cite[p. 214, 15.22]{M1}).
\end{proof}

To continue the proof recall that
$$
\operatorname{Cap}(F)^{\frac{1}{d-2}} \ge c_{d} |F|^{\frac{1}{d}},
$$
where $|F|$ denotes the $d$-dimensional Lebesgue measure of the compact set~$F$. Therefore \eqref{eq55} yields, at almost all points~$a \in E$ and for all $t>0$,
$$
\lim_{r\to 0}\frac{1}{|B(a,r)|} \left|\left\{x\in B(a,r)\cap E:\frac{|\langle B(x-a),x-a \rangle|}{|x-a|^{2}}>t\right\}\right| =0.
$$
Set $a=0$ and
$$
U=U_{t}=\left\{x\in\mathbb{R}^{d}\backslash \{0\}:\frac{|\langle Bx,x\rangle |}{|x|^{2}}>t\right\}.
$$
Then
\begin{equation*}
\begin{split}
\frac{|B(0,1)\cap U|}{|B(0,1)|}&=\frac{|B(0,r)\cap U|}{|B(0,r)|}=\frac{|B(0,r)\cap U\cap E|}{|B(0,r)|} +\frac{|B(0,r)\cap U\cap E^{c}|}{|B(0,r)|}\\*[7pt]
&\le \frac{|B(0,r)\cap U\cap E|}{|B(0,r)|}+\frac{|B(0,r)\cap E^{c}|}{|B(0,r)|}.
\end{split}
\end{equation*}
If $0$ is a point of density of~$E$ we obtain, letting $r \rightarrow 0$, that $|U|=0$, which means, $U$ being an open set, that $U=U_{t}=\emptyset$ for all~$t$. In other words, $B\equiv 0$ and thus, appealing to~\eqref{eq54}, $f(a)=0$, for almost all $a \in E$.
\end{proof}

\begin{proof}[Proof in $\mathbb{R}^{2}$]
We plan to apply Theorem~\ref{teo4}.  Recall that in the plane when dealing with capacity we tacitly assume that all our sets are contained in the disc centered at the origin and of radius $1/2$. Since the equilibrium potential is constant $\operatorname{Cap}$-a.e. on~$E$  we have for some real numbers~$A_{1}$ and~$A_{2}$
$$
\lim_{r\to 0}\frac{1}{\operatorname{Cap}B(a,r) } \sup\limits_{t>0} t \operatorname{Cap}\biggl(\biggl\{x\in B(a,r)\cap E: \frac{\left|\sum\limits_{i=1}^{2} A_{i}(x_{i}-a_{i})\right|}{|x-a|}>t\biggr\}\biggr) =0
$$
at $H^{\psi}$-almost all points $a\in E$, hence at almost all points $a\in E$ with respect to area. Set $\nabla u(a)=(A_{1},A_{2})$. By the well-known inequality \cite[Corollary 5.1.14]{AH}
$$
\operatorname{Cap}(F)\ge C\, \frac{1}{\log (1 / \operatorname{H^1_\infty}(F))},
$$
valid for a constant $C$~independent of the compact set~$F$, we get, for each $t > 0$,
$$
\lim_{r\to 0}\frac{1}{r}  H^{1}_{\infty}\left(\left\{ x\in B(a,r)\cap E: \frac{|\langle \nabla u(a),x-a\rangle|}{|x-a|}>t\right\}\right)=0.
$$
Assume that $\nabla u(a)\ne 0$, set $a=0$ and, without loss of generality, $\nabla u(0)=\lambda (0,1)$, $\lambda>0$. Then we obtain, with $\delta =t$,
$$
\lim_{r\to 0}\frac{H^{1}_{\infty}(B(a,r)\cap E\cap K_{\delta})}{r}=0,
$$
where $K_{\delta}$ is the cone~\eqref{eq56}. Hence the line~$x_{2}=0$ is an approximate tangent line for~$E$ at~$0$. Therefore the set of points in~$E$ where $\nabla u(a)$ is non-zero is a countable union of sets of finite length. In particular $\nabla u(a)=0$, for almost all $a \in E$ and
$$
C\mu(a)=\left(\frac{1}{z}*\mu\right)(a)=0,\quad \text{a.e.\ on }E.
$$
We can now resort to the proof of Theorem~1 in \cite{TV} to conclude that the absolutely continuous part of~$\mu$ vanishes. Indeed in \cite{TV} one takes $\mu$ absolutely continuous with respect to~$dx$, but a minor variation of the argument applies to our situation.
\end{proof}

\begin{gracies}
We are grateful to S. Gardiner for some useful correspondence. The first named author was partially supported by the grant 2014SGR289
(Ge\-ne\-ra\-li\-tat de Catalunya).
The second named author was partially supported by the grants 2014SGR75
(Generalitat de Catalunya) and  MTM2013--44699 (Mi\-nis\-terio de
Educaci\'{o}n y Ciencia).
\end{gracies}

\vspace*{.55cm}

\begin{tabular}{l}
Juli\`{a} Cuf\'{\i} and Joan Verdera\\
Departament de Matem\`{a}tiques\\
Universitat Aut\`{o}noma de Barcelona\\
08193 Bellaterra, Barcelona, Catalonia\\
{\it E-mail:} {\tt jcufi@mat.uab.cat}\\
{\it E-mail:} {\tt jvm@mat.uab.cat}
\end{tabular}

\end{document}